%% file: main.tex
\documentclass{article}
\usepackage{graphicx} 

\usepackage{authblk}

\usepackage{geometry}
\usepackage{amsmath}
\usepackage{comment}
\usepackage{leftidx}
\usepackage{mathtools}
\usepackage{multicol}
\usepackage{amssymb}

\usepackage{stmaryrd}
\usepackage{accents}
\usepackage{mathrsfs}
\usepackage{xfrac}
\usepackage{amsthm}
\usepackage{xcolor}
\usepackage{sectsty}
\usepackage{relsize}
\usepackage{float}

\usepackage{Mypackage}

\usepackage{musicography}

\NewDocumentCommand{\mustrill}{}{\musSymbol{0.05em}{0.5ex}{1.35em}{\symbol{87}}}

\usepackage[all]{xy}
\usepackage[T1]{fontenc}
\usepackage[utf8]{inputenc}

\usepackage{indentfirst}
\usepackage{tikz}
\tikzset{shorten <>/.style={shorten >=#1,shorten <=#1}}

\usepackage{tikz-cd}
\newcounter{nodemaker}
\tikzset{Rightarrow/.style={double equal sign distance,>={Implies},->},
triple/.style={-,preaction={draw,Rightarrow}},
quadruple/.style={preaction={draw,Rightarrow,shorten >=0pt},shorten >=1pt,-,double,double
distance=0.2pt}}

\tikzset{%
    symbol/.style={%
        draw=none,
        every to/.append style={%
            edge node={node [sloped, allow upside down, auto=false]{$#1$}}}
    }
}
 \usepackage{quiver}
\usetikzlibrary{nfold}

\usepackage{hyperref}
\hypersetup{
    colorlinks = true,
    linkbordercolor = {red},
   	linkcolor ={teal},
	anchorcolor = {pink},
	citecolor =  {orange},
	filecolor = {teal},
	menucolor = {teal},
	runcolor =  {teal},
	urlcolor = {teal},
}

\usepackage{cleveref}
\newtheorem{theorem}{Theorem}[subsection]
\newtheorem*{theorem*}{Theorem}
\newtheorem{proposition}[theorem]{Proposition}
\newtheorem{corollary}[theorem]{Corollary}
\newtheorem{corollary'}[theorem]{Corollary}

\newtheorem{lemma}[theorem]{Lemma}
\theoremstyle{definition}
\newtheorem{definition}[theorem]{Definition}
\newtheorem{example}[theorem]{Example}
\theoremstyle{definition}
\newtheorem{remark}[theorem]{Remark}
\theoremstyle{definition}
\newtheorem{division}[theorem]{}
\theoremstyle{definition}

\theoremstyle{definition}

\theoremstyle{definition}

\usepackage{imakeidx}
\usepackage{authblk}

\usepackage[backend=biber,
sorting=nty
]{biblatex}
\title{Morphisms and comorphisms of sites II \\ Distributors of sites}
\author[1]{Olivia Caramello}
\author[1]{Axel Osmond}
\affil[1]{Istituto Grothendieck}
\date{}

\begin{document}

\maketitle

\begin{abstract}
    We introduce a notion of distributor of sites, involving suited analogs of flatness and cover-preservation, and show that this notion jointly generalizes those of morphism and comorphism of sites. Given two sites, we exhibit an adjunction between the category of distributors of sites between them and the category of geometric morphisms between the associated sheaf topoi; this adjunction restricts to an equivalence between geometric morphisms and continuous distributors of sites. We finally discuss some equipment-like properties of the bicategory of sites and distributors of sites. 
\end{abstract}

\tableofcontents

\newpage

\input{Sections/Intro}

\newpage

\input{Sections/Distributeurs_sites}

\printbibliography

\end{document}

%% file: Sections/Intro.tex
\section*{Introduction}

This work is the second installment in a program focused on the relations between morphisms and comorphisms of sites, and in particular the problem of mixing them altogether inside a single 2-dimensional structure.

In the first part \cite{caramello2025morphisms}, we introduced a double category of sites, morphisms and comorphisms, and then we exhibited sheafification as a double functor to the lax quintet double category of topoi. This second part presents a different solution: we show that morphisms and comorphisms can jointly be generalized by a single notion of \emph{distributor of sites}. Those are distributors that combine appropriate analogs of the usual properties of cover-flatness and cover-preservation involved in the definition of morphisms of sites. 

A notion of flatness for distributors between categories was first introduced in \cite{benabou2000distributors}; we describe at \cref{cover-flat distributor} a refinement of this notion in the presence of coverages. On the other hand, we introduce at \cref{cover-distribution} a condition of \emph{cover-distribution}, expressing a form of joint-preservation of covers along heteromorphisms. This notion happens to be closely related to a notion isolated in \cite{johnstonewraith}, which was however stated in terms of bidense morphisms. 

We then explain in \cref{distributor of sites induce geomorph} how distributors of sites induce geometric morphisms; however, a notable improvement relative to morphisms or comorphisms is that now, given two sites, any geometric morphism between the associated sheaf topoi can be induced by a distributor between those sites -- see \cref{adjunction} -- while one may fail to exhibit a functorial lift if both sites are fixed. This process restricts to an equivalence between geometric morphisms and distributors of sites satisfying a further continuity condition, as stated in \cref{continuous distributors of sites corresponds with geometric morphisms}. 

It should be emphasized that the profunctorial aspects here-examined do not fit in the double-categorical framework investigated in our first part \cite{caramello2025morphisms}, where the double-categories of interest were of a very different kind from those usually involved in \emph{pro-arrow equipments}. 

Although this second part makes no use of any double-categorical technology, we nevertheless describe in our last section some equipment-like properties of the bicategory of sites together with distributors of sites; while not all axioms of equipments are fulfilled -- because of the very dichotomy between morphisms and comorphisms amongst other reasons -- it makes sense to expect some similarities with the equipment provided by distributors, for instance well behaved gluing as described at \cref{gluing}. A further third part of our program with be devoted to the formal category theory for sites.  

%% file: Sections/Distributeurs_sites.tex
\section{Distributors between categories}

In this section, we recall some generalities about distributors, and how in particular they induce geometric morphisms between presheaf topoi in a covariant way. We devote also a subsection to the lesser known \cite{benabou1975fibrations} notion of \emph{flat} distributors, which have enough property to induce geometric morphisms between corresponding presheaf topoi in a contravariant way. 

\input{Sections/Prequisite}

There is another way to mix morphisms and comorphisms, this time as two instances of a same notion of 1-cells. 
It also relates to the imperfect correspondence of (co)morphisms of sites with geometric morphisms: 
for two fixed sites $ (\mathcal{C},J)$, $ (\mathcal{D},K)$, there are geometric morphisms $ \widehat{\mathcal{D}}_K \rightarrow \widehat{\mathcal{C}}_J $ that are not induced from morphism of sites $ (\mathcal{C},J) \rightarrow (\mathcal{D},K)$, for the inverse image may not restrict to a functor. 
But it always restricts to \emph{a distributor}, and we may ask what flatness and continuity mean for distributors.

\subsection{Generalities on distributors}

Distributors are relational analog of functors;  

\begin{definition}
A \emph{distributor} $ H : \mathcal{C} \looparrowright \mathcal{D}$ is a functor $ \mathcal{D}^{\op} \times \mathcal{C} \rightarrow \Set$. A \emph{transformation of distributors} $ \phi : H \Rightarrow H'$ is a natural transformation in $ \Cat[\D^{\op} \times \C, \Set]$. 
\end{definition}

For a pair $ (d,c)$ with $d$ in $\D$ and $c $ in $\C$, the elements of $H(d,c)$ are called \emph{heteromorphisms from $d$ to $c$} and will be denoted as generalized arrows across categories as below
\[\begin{tikzcd}
	d & c
	\arrow["x", squiggly, from=1-1, to=1-2]
\end{tikzcd}\]

This terminology echoes to the formal name of \emph{homomorphism} for ordinary arrows inside a category; those are in fact a special case of heteromorphisms, namely those associated to the \emph{unit} $ \hom_\C : \C \looparrowright \C$ given by the hom functor $ \C[-,-] : \C^\op \times \C \rightarrow \Set$.

\begin{division}[Composition of distributors]
    For two distributors $ H : \mathcal{C} \looparrowright \mathcal{D}$ and $ K : \mathcal{D} \looparrowright \mathcal{E}$, one can consider the composite distributor, also known as their \emph{tensor product} $ K \otimes H : \mathcal{C} \looparrowright \mathcal{E}$, which is computed at a pair $ (e,c)$ as the coend
\[ (K \otimes H)[e,c] = \int^{d \in \mathcal{D}} H[d,c] \times K[e,d] \]



    The unit for the composition of distributors is given by the homset distributor $\hom_\C : \C \looparrowright \C$ sending $ (c,c')$ to $ \C[c,c']$. 
\end{division}

\begin{remark}
    Beware that composition is weak: associativity and unit laws only hold up to canonical invertible 2-cells. For this reason, categories together with distributors and transformations between them actually arrange into a \emph{bicategory} rather than a strict 2-category. 
\end{remark}

We denote as $\Dist$ the bicategory of category, distributors and transformations between them.

\begin{division}[Distributors and free cocompletion]
    Recall that for two categories $ \mathcal{C}, \mathcal{D}$, the category of distributors is in equivalence with the category of cocontinuous functors:
    \begin{align*}
        \Dist[\mathcal{C}, \mathcal{D}] &= \CAT[\mathcal{D}^{\op} \times \mathcal{C}, \Set] \\
        &\simeq \CAT[\mathcal{C}, \widehat{\mathcal{D}}] \\
        &\simeq \coCont[\widehat{\mathcal{C}}, \widehat{\mathcal{D}}]
    \end{align*}
    where a distributor $H : \mathcal{C} \looparrowright \mathcal{D}$ is sent to the functor $ \widehat{H} : \mathcal{C} \rightarrow \widehat{\mathcal{D}}$ sending $ c$ to the presheaf $ H(-,c)$ on $\mathcal{D}$, which in turn induces a cocontinuous functor $\lext_H : \widehat{\mathcal{C}} \rightarrow \widehat{\mathcal{D}}$ obtained through the left Kan extension $\lext_H = \lan_{\hirayo_\mathcal{C}} \widehat{H} $. Concretely, this cocontinuous functor returns at a presheaf $ X : \mathcal{C}^{\op} \rightarrow \Set$ the presheaf on $\mathcal{D}$ computed at an object $d$ as the coend:
        \[ \lext_H(X)(d) = \int^{c \in \mathcal{C}} H(d,c) \times X(c)  \] 
    Moreover, being cocontinuous between presheaf categories, $ \lext_H$ possesses a right adjoint $ \rest_H = \lan_H \hirayo_\mathcal{D}$ sending a presheaf $ Y : \mathcal{D}^{\op} \rightarrow \Set$ to the presheaf on $\mathcal{C}$ computed at an object $c$ as the homset
\[ \rest_H(Y)(c) = \widehat{\mathcal{D}}[\widehat{H}(c), Y] \]
\end{division}

\begin{remark}\label{lext_H has no radj}
    In the case of a functor $f : \C \rightarrow \D$, we saw at \cref{Formulas for restrictions and lextensions} that the restriction functor expresses at a presheaf $Y$ in $\widehat{\D}$ as the nerve
    \[ \rest_f(Y) \simeq \widehat{\mathcal{D}}[\hirayo_{\mathcal{D}}f, Y]  \]
    where the object on the left, being representable, is indecomposable in $\widehat{\D}$, hence preserves all colimits: this is why $ \rest_f$ preserves colimits, hence admits a further right adjoint $ \rext_f : \widehat{\C} \rightarrow \widehat{\D}$. However in the case of a distributor $ H : \C \rightarrow \D$, it is not true anymore in general that $ \widehat{H}(c)$ needs being representable, nor even indecomposable. Hence in full generality, $ \rest_H$ needs \emph{not} be cocontinuous, and may lack a further right adjoint. 
\end{remark}


\begin{division}
    The construction of the classifying functor associated to a distributor is pseudofunctorial. For a pair $ H : \C \looparrowright \D$, $ K : \D \looparrowright \E$ one has 
    \begin{align*}
        \widehat{K\otimes H} &= \lext_K \circ \widehat{H} 
    \end{align*}
  Moreover, as observed above, the classifying functor of the unit distributor $ \widehat{\hom_\C} : \C \to \widehat{\C} $ is precisely the Yoneda embedding, and by denseness of this latter, the corresponding geometric morphism $ \widehat{\C} \rightarrow \widehat{\C}$ is equivalent to the identity $ 1_{\widehat{\C}}$.

  If now one considers a transformation of distributors $ \phi : H \Rightarrow K$, that is, a natural transformation with components $ \phi_{(d,c)} : H(d,c) \rightarrow K(d,c)$, we can induce a natural transformation $ \widehat{\phi} : \widehat{H} \Rightarrow \widehat{K}$ whose component at $ c$ is the morphism of presheaves $ H(-,c) \Rightarrow K(-,c)$ whose evaluation at $d$ is given by $\phi_{(d,c)}$.  
\end{division}

\begin{division}[Representables]\label{representable distributors}
    Any functor $ f : \mathcal{C} \rightarrow \mathcal{D}$ induces two distributors: \begin{itemize}
        \item the representable $ \mathcal{D}(1,f) : \mathcal{C} \looparrowright \mathcal{D}$ sending $ (d,c)$ to $ \mathcal{D}(d,f(c))$,
        \item the corepresentable $ \mathcal{D}(f,1) : \mathcal{D} \looparrowright \mathcal{C}$ sending $ (c,d)$ to $ \mathcal{D}(f(c),d)$.
    \end{itemize}  
    For any $f : \C \rightarrow \D$, we have an adjunction $ \D[1, f] \dashv \D[f,1]$ internally to $ \Dist$, as explained in \cite{borceux1994handbook1}[proposition 7.9.1]. Conversely, if $ \C$, $\D$ are Cauchy-complete, then any distributor admitting a right adjoint in $ \Dist$ is of the form $ \D(1,f)$ for some functor $ f : \C \rightarrow \D$. 
\end{division}


\begin{division}[Extensions of representable distributors]\label{Extensions of representable distributors}
Suppose one has a functor $ f : \mathcal{C} \rightarrow \mathcal{D}$: consider the associated distributor $ \mathcal{D}(1,f) : \mathcal{C} \looparrowright \mathcal{D}$. The functor $ \mathcal{C} \rightarrow \widehat{\mathcal{D}}$ induced from this distributor is precisely the functor $ \mathcal{D}(1,f) = \hirayo_{\D} f$ as defined at \cref{Extensions and restrictions}. Then from \cref{Formulas for restrictions and lextensions} we know that the left extension of $ \mathcal{D}(1,f)$ as a distributor coincides with the left extension functor:
\[ \lext_{\mathcal{D}(1,f)} = \lext_f \]

On the other hand, suppose one has a functor $ f: \mathcal{D} \rightarrow \mathcal{C}$ and take the other associated distributor $ \mathcal{C}(f,1) : \mathcal{C} \looparrowright \mathcal{D} $. The functor $ \mathcal{C} \rightarrow \widehat{\mathcal{D}}$ induced from this distributor is precisely the nerve functor $ \mathcal{C}(f,1)$ sending $ c$ to the presheaf $ \mathcal{C}[f,c] : \mathcal{D}^{\op} \rightarrow \Set$. Then, the expression of the associated cocontinuous functor as the left Kan extension of this induced functor, together with the first expression of the restriction functor given at \cref{Formulas for restrictions and lextensions}, gives that 
\[ \lext_{\mathcal{C}(f,1)} = \rest_f \]  
As we will see, those formulas will play an important role when unifying cover-preserving and cover-lifting conditions into a same condition for distributors. 
\end{division}

\begin{division}
    Distributors are classified by the presheaf construction in the sense that for any $\C, \D$, we have an equivalence of categories $ \Dist[\C,\D] \simeq \CAT[\C,\widehat{\D}]$. Formally, this means that $ \Dist$ is the \emph{Kleisli 2-category} for the relative pseudomonad $ \widehat{(-)} : \Cat \rightarrow \CAT$. We will exhibit latter a similar property for an appropriate bicategory of continuous distributors between sites. 
\end{division}


\subsection{Flat and lex distributors}

As well as flatness is required to construct a geometric morphism from a functor, we will require an analog of flatness for distributors; this will be sufficient to capture all geometric morphisms between associated presheaf topoi. Then content of this subsection is closely related to \cite{benabou2000distributors}[section 6] where the idea of flat distributor is introduced, while they also briefly appear in the context of torsors at \cite{elephant}[B3.2.8].

\begin{division}[Opfibration of elements of a distributor]
A distributor $ H : \mathcal{C} \looparrowright \mathcal{D}$ induces in each $d$ of $\mathcal{D}$ a functor $ H(d,-) : \mathcal{C} \rightarrow \Set$ sending $ c$ to $H(d,c)$ and $ u :c \rightarrow c'$ to the transition functor $ H(d,u) : H(d,c) \rightarrow H(d,c')$. This induces a discrete opfibration $ \pi^H_d : d \downarrow H \rightarrow \mathcal{C}$ where $ d \downarrow H$ has as objects $ (c,x)$ with $x \in H(d,c)$ and morphisms $(c,x) \rightarrow (c',x')$ those $u : c \rightarrow c'$ such that $H(d,u)(x) = x'$. Then in the extension above, the coend formula for the left extension $ \lext_H$ re-expresses as a colimit of the projection
\[  \lext_H(X)(d) = \colim ( (d \downarrow H)^{\op} \rightarrow \mathcal{C}^{\op} \rightarrow \Set ) \]
\end{division}

\begin{definition}
Then a distributor $H : \mathcal{C} \looparrowright \mathcal{D} $ is said to be \emph{flat} if for each $d$ the category $ d \downarrow H$ is cofiltered.     
\end{definition}

\begin{remark}
This generalizes exactly the usual notion of representable flatness; it means that for any $d$ we have\begin{enumerate}
    \item $d \downarrow H$ is non empty, so there is an heteromorphism $ d \rightsquigarrow c $
    \item for any diagram as below
\[\begin{tikzcd}
	& c \\
	d \\
	& {c'}
	\arrow["x", squiggly, from=2-1, to=1-2]
	\arrow["{x'}"', squiggly, from=2-1, to=3-2]
\end{tikzcd}\]
there is a span
\[\begin{tikzcd}
	& c \\
	d & {c''} \\
	& {c'}
	\arrow["x", squiggly, from=2-1, to=1-2]
	\arrow["{x''}"{description}, squiggly, from=2-1, to=2-2]
	\arrow["{x'}"', squiggly, from=2-1, to=3-2]
	\arrow["u"', from=2-2, to=1-2]
	\arrow["{u'}", from=2-2, to=3-2]
\end{tikzcd}\]
\item and for any diagram as below -- that is, $x \in H(d,c)$ such that $ H(d,u) (x) = x' = H(d,u')(x)$
\[\begin{tikzcd}
	& {c'} \\
	d & c
	\arrow["{x'}", squiggly, from=2-1, to=1-2]
	\arrow["x"', squiggly, from=2-1, to=2-2]
	\arrow["u", shift left, from=2-2, to=1-2]
	\arrow["{u'}"', shift right, from=2-2, to=1-2]
\end{tikzcd}\]
there is $ u'' : c'' \rightarrow c$ with $ uu'' = u'u''$ and $ x''$ such that $u''(x'') = x$
\[\begin{tikzcd}
	& {c'} \\
	d & c \\
	& {c''}
	\arrow["{x'}", squiggly, from=2-1, to=1-2]
	\arrow["x"{description}, squiggly, from=2-1, to=2-2]
	\arrow["{x''}"', squiggly, from=2-1, to=3-2]
	\arrow["u", shift left, from=2-2, to=1-2]
	\arrow["{u'}"', shift right, from=2-2, to=1-2]
	\arrow["{u''}"', from=3-2, to=2-2]
\end{tikzcd}\]
\end{enumerate}     
\end{remark}

The following is a rephrasing of a fact established in \cite{benabou2000distributors}[theorem 6.5] as well as \cite{elephant}[B.3.2.8 (c)]:

\begin{proposition}
 
For a distributor $ H : \mathcal{C} \looparrowright \mathcal{D}$ the following are equivalent:\begin{itemize}
    \item $H$ is a flat distributor;
    \item the classifying functor $ \widehat{H} : \C \rightarrow \widehat{\D}$ is representably flat;
    \item the left extension $ \lext_H : \widehat{\mathcal{C}} \rightarrow \widehat{\mathcal{D}}$ preserves finite limits.
\end{itemize}    
\end{proposition}

\begin{corollary}
A distributor $ H $ is flat if and only if it induces a geometric morphism $ (\lext_H, \rest_H) : \widehat{\mathcal{D}} \rightarrow \widehat{\mathcal{C}}$. Conversely, any geometric morphism $ (f^*, f_*) : \widehat{\mathcal{D}} \rightarrow \widehat{\mathcal{C}}$ is uniquely induced by a flat distributor $ H_f : \mathcal{C} \looparrowright \mathcal{D}$.      
\end{corollary}

\begin{proof}
Consider indeed the distributor $H_f$ sending $ (d,c)$ to the homset $ \widehat{\mathcal{D}}[\hirayo_d, f^*\hirayo_c]$, that is, to the valued of the presheaf $ f^*\hirayo_c(d)$. This distributor is flat because the restriction $ f^*\hirayo : \mathcal{C} \rightarrow \widehat{\mathcal{D}}$ is representably flat for $f^*$ is so.    
\end{proof}

The pseudofunctor $\widehat{(-)} : \Dist \rightarrow \coCont $ induces a pseudofunctor $ (\Flat\Dist)^{\op}\rightarrow \Top$:

\begin{proposition}
The pseudofunctor $ \widehat{(-)} : \Flat\Dist^{\op} \rightarrow \Top$ is pseudofully faithful, that is, for any $ \mathcal{C}, \mathcal{D}$ we have an equivalence of categories
\[  \Flat\Dist[\mathcal{C}, \mathcal{D}] \simeq \Top[\widehat{\mathcal{D}}, \widehat{\mathcal{C}}] \]
\end{proposition}

\begin{proof}
    We prove the functors above to be mutual inverse. For a distributor $H : \mathcal{C} \looparrowright \mathcal{D}$, one has in each $(d,c)$ 
    \begin{align*}
        H_{\lext_H}(d,c) &\simeq \lext_H(\hirayo_c)(d) \\
        &\simeq \int^{c' \in \mathcal{C}} H(d,c') \times \mathcal{C}[c',c] \\
        &\simeq H(d,c)
    \end{align*}
    where the last isomorphism comes from the coend decomposition of the functor $ H(d,-)$ under co-Yoneda. In the converse direction, for a geometric morphism $f : \widehat{\mathcal{D}} \rightarrow \widehat{\mathcal{C}}$ one has at each object $E$ in $\widehat{\mathcal{C}}$ and $d$ in $\mathcal{D}$:
    \begin{align*}
        \lext_{H_f}E(d) &\simeq \int^{c \in \mathcal{C}} f^*\hirayo_c(d) \times E(c) \\ 
        &\simeq \colim_{c \in \mathcal{C}^{\op}}^{f^*\hirayo_\mathcal{C}(d)} E(c) \\
        &\simeq (\lan_{f^{\op} \mathfrak{l}_{(\mathcal{C},J)}^{\op}} E)(d)\\
        &\simeq f^*E(d)
    \end{align*}
\end{proof}

Flat functors generalize lex functor between lex categories; but if $\C$, $\D$ are lex categories, the property for $f : \C \rightarrow \D$ to be lex amounts to having for each finite diagram $ (c_i)_{ i\in I}$ in $\C$ a natural isomorphism at each $d$ 
\[ \D[d, f(\lim_{i \in I} \, c_i)] \simeq \lim_{i\in I} \, H(d,c_i) \]

This leads to the following notion:

\begin{definition}
    A distributor $ H : \C\looparrowright \D$ between lex categories is said to be \emph{lex} if for any finite diagram $ (c_i)_{i \in I}$ in $\C$ a natural isomorphism at each $d$ 
    \[ H(d,\lim_{i \in I} \, c_i) \simeq \lim_{i\in I} \, H(d,c_i) \]
\end{definition}

\begin{proposition}
    A distributor $ H : \C\looparrowright \D$ between lex categories is flat if and only if it is lex.
\end{proposition}

\begin{proof}
    By what precedes, $H$ is flat if and only if $\widehat{H}$ is representably flat, which reduces to being lex by lexness of $\C$, which amounts to lexness of $H$ itself by pointwiseness of limits in $\widehat{\D}$.
\end{proof}

\begin{remark}\label{lex dist induce normal geom}
    Then in particular if $\C$, $\D$ are small lex categories, one get an equivalence of categories, denoting $\Lex\Dist[\C,\D]$ the category of lex distributors between $\C$, $\D$, 
    \[ \Lex\Dist[\C,\D] \simeq \Top[\widehat{\D}, \widehat{\C}] \]
    
    This generalizes the fact that a lex functor between lex categories $ f : \C \rightarrow \D$ induces more specifically a \emph{local} geometric morphism $ (\lext_f \dashv \rest_f \dashv \rext_f) : \widehat{\D} \rightarrow \widehat{\C}$ with exceptional image part given by $\rext_f$.

    Now in the case of a distributor, we saw at \cref{lext_H has no radj} that differently from the functorial case, the restriction $ \rest_H$ needs not having a further right adjoint: hence, in the case where $H$ is flat, the geometric morphism $ (\lext_H \dashv \rest_H) : \widehat{\D} \rightarrow \widehat{\C}$ needs not being local, nor in fact, enjoying any specific property at all ! 
\end{remark}

\section{Distributors of sites}

In this section we describe how distributors between sites satisfying appropriate conditions are in bijective correspondence with geometric morphisms between associated sheaf topoi, enhancing the imperfect correspondence with morphisms of sites or comorphisms of sites.

\subsection{Cover-distributivity and cover-flatness}

Let us first re-express the conditions of being a morphism and a comorphism of site in terms of heteromorphisms, using the extension and restriction formalism. 
\begin{division}[Restrictions and extensions on sieves]
Let $ f: \mathcal{C} \rightarrow \mathcal{D}$ a functor and let be $ J$ and $K$ topologies on $\mathcal{C}$ and $ \mathcal{D}$: then one can define\begin{itemize}
    \item for each $J$-covering sieve $S$ on $c$ in $\mathcal{C}$, the sieve $ f[S]$ containing all $ x : d \rightarrow f(c)$ that factorize through an arrow $f(u) : f(c') \rightarrow f(c)$ with $u : c' \rightarrow c$ in $S$: this is the image factorization of the lextension 
    \[\begin{tikzcd}
	{\lext_f(S)} && {\hirayo_{f(c)}} \\
	& {f[S]}
	\arrow["{\lext_f(m)}", from=1-1, to=1-3]
	\arrow[two heads, from=1-1, to=2-2]
	\arrow["{f[m]}"', tail, from=2-2, to=1-3]
\end{tikzcd}\]
    \item for each $K$-covering sieve $ R$ on $f(c)$, the sieve $ f^{-1}(R)$ on $c$ containing all $ u: c' \rightarrow c$ such that $f(u)$ is in $R$. This is the pullback along the name of the unit of the restriction
    \[\begin{tikzcd}
	{f^{-1}(R)} & {\rest_f R} \\
	{\hirayo_c} & {\mathcal{D}(f,f(c))}
	\arrow[from=1-1, to=1-2]
	\arrow[tail, from=1-1, to=2-1]
	\arrow["\lrcorner"{anchor=center, pos=0.125}, draw=none, from=1-1, to=2-2]
	\arrow[tail, from=1-2, to=2-2]
	\arrow["{\nu_c}"', from=2-1, to=2-2]
\end{tikzcd}\]
\end{itemize}
\end{division}

\begin{division}[Generalized characterization of morphisms of sites]
A functor $f : \mathcal{C} \rightarrow \mathcal{D}$ defines
a morphism of sites if for any sieve $ S \rightarrowtail \hirayo_c$ the image sieve $ f(S)$, that is, the sieve generated by the arrows of the form $f(u)$ for $ u : c' \rightarrow c$ in $S$ (this is $ f[S]$) is to be $ K$-covering -- for which it is enough it contains a $K$-covering sieve. But if $ f[S]$ is covering, then for any $x : d \rightarrow f(c)$ the pullback sieve $ x^*f[S]$ consisting of all those $v : d' \rightarrow d$ that factorize through some $ f(u) : f(c') \rightarrow f(c)$ must be covering. This is exactly the set of all $ v$ such that there is a $(u,x')$ as below of heteromorphisms for $ \mathcal{D}(1,f)$
\[\begin{tikzcd}
	d & c \\
	{d'} & {c'}
	\arrow["x", squiggly, from=1-1, to=1-2]
	\arrow["v", from=2-1, to=1-1]
	\arrow["{x'}"', squiggly, from=2-1, to=2-2]
	\arrow["u"', from=2-2, to=1-2]
\end{tikzcd}\]

\end{division}

Then we can rephrase in a slightly more general way the conditions of being cover preserving and cover-lifting. 

\begin{proposition}\label{morphisms of sites in term of lext}
    For a functor $f : \mathcal{C} \rightarrow \mathcal{D}$, the following are equivalent:\begin{itemize}
        \item $f$ is a cover-preserving functor $(\mathcal{C},J) \rightarrow (\mathcal{D},K)$
        \item for each $c$ in $\mathcal{C}$, each $J$-covering sieve $ S$ over $c$ and each $x : d \rightarrow f(c) $ in $ \mathcal{D}(1,f)(d,c)$, the pullback sieve $ x^*f[S]$ is in $K(d)$.
    \end{itemize}  
\end{proposition}

\begin{proof}
In one direction suppose $ f$ to be a cover-preserving, and take $c$ and $S$ on $c$. If $ S$ is $J$-covering, then for $f$ is a morphism, $x^*f[S]$ is $K$-covering; but then for any $d $ in $\mathcal{D}$ and $x : d \rightarrow f(c)$, by stability, $x^*f[S]$ is in $K(d)$.

Conversely, suppose that for each $ x : d \rightarrow f(c)$ and $S$ in $J(c)$ as above, $ x^*f[S]$ is $K(d)$ covering: then in particular $ \id_{f(c)}^*f[S] = f[S]$ is $K(f(c))$-covering, so that $ f$ sends $J$-covering sieves to $ K$-covering sieves.\end{proof}

\begin{division}[Generalized characterization of comorphisms of sites]\label{comorphism with rest}
Dually, a functor $f : \mathcal{C} \rightarrow \mathcal{D}$ defines
a comorphism of sites $(\mathcal{C},J) \rightarrow (\mathcal{D},K)$ if
for any $K$-covering sieve $ R $ over $f(c)$, the sieve $ f^{-1}(R)$ of all $ u : c' \rightarrow c$ in $\mathcal{C}$ such that $ f(u)$ is in $R$ is $J$-covering: in other words, if $R$ is $K$-covering on $f(c)$ then there is an $S$ on $c$ in $J$ such that $ S \subseteq f^{-1}(R)$, equivalently $ f[S] \subseteq R$. 



In this case we use the dual distributor $ \mathcal{D}(f,1)$: then if $f$ is a comorphism, for any $ x: f(c) \rightarrow d$ and any $K$-sieve $R$ on $d$, as the pullback-sieve $ x^*R$ is a $K$-covering sieve on $f(c)$, the cover-lifting property requires that $ f^{-1}(x^*R)$ is $J$-covering: but the latter consists in all the $u: c' \rightarrow c$ for which there is a pair $ (v,x')$ and a factorization 
\[\begin{tikzcd}
	{f(c)} & d \\
	{f(c')} & {d'}
	\arrow["x", from=1-1, to=1-2]
	\arrow["{f(u)}", from=2-1, to=1-1]
	\arrow["{x'}"', dashed, from=2-1, to=2-2]
	\arrow["v"', from=2-2, to=1-2]
\end{tikzcd}\]
which is written in terms of heteromorphisms again as 
\[\begin{tikzcd}
	d & c \\
	{d'} & {c'}
	\arrow["x", squiggly, from=1-1, to=1-2]
	\arrow["v", from=2-1, to=1-1]
	\arrow["{x'}"', squiggly, from=2-1, to=2-2]
	\arrow["u"', from=2-2, to=1-2]
\end{tikzcd}\]
This is in fact the pullback of the preimage sieve along the name of $x$, as depicted below:
\[\begin{tikzcd}
	{f^{-1}(x^*S)} & {\rest_f(x^*S)} \\
	{\hirayo_d} & {\C(f,c)}
	\arrow[from=1-1, to=1-2]
	\arrow[tail, from=1-1, to=2-1]
	\arrow["\lrcorner"{anchor=center, pos=0.125}, draw=none, from=1-1, to=2-2]
	\arrow["{\rest_f(x^*m)}", tail, from=1-2, to=2-2]
	\arrow["x"', from=2-1, to=2-2]
\end{tikzcd}\]
where the restriction $ \rest_f(m)$ is a monomorphism for $\rest_f$ preserves monomorphisms as a right adjoint. 
\end{division}

\begin{proposition}\label{comorphism of sites in term of rest}
A functor $f : \mathcal{C} \rightarrow \mathcal{D}$, the following are equivalent:\begin{itemize}
        \item $f$ defines a cover-lifting functor $(\mathcal{C},J) \rightarrow (\mathcal{D},K)$
        \item for each $d$ in $\mathcal{D}$, each $K$-covering sieve $ R$ over $d$ and each $x : f(c) \rightarrow d$ in $ \mathcal{D}(f,1)(d,c)$, the preimage sieve $ f^{-1}(x^*R)$ is in $J(c)$.
    \end{itemize}  
\end{proposition}

Having in mind those heteromorphic formulations of cover-preservation and cover-lifting, we are close to exhibit them as two manifestation of a single notion stated in term of distributors. To properly state it, it remains to describe the action of a distributor relative to sieves as we did in the case of functors.

\begin{division}[Left extensions of sieves]
    Let $ H : \mathcal{C} \looparrowright \mathcal{D}$ be a distributor between two sites $(\mathcal{C},J)$ and $(\mathcal{D},K)$. For a sieve $ S \rightarrowtail \hirayo_c $ on an object of $\mathcal{C}$, the left extension of the distributor $ H$ returns at any object $d$ the coend 
    \[  \lext_H(S)(d) = \int^{c' \in \mathcal{C}} H(d,c') \times S(c') \]
    which is the set of all pairs $ (c',x', u') $ with $ x' \in H(d,c')$ and $u \in S(c')$ where one identifies pairs $ (c',x') $ and $(c'',x'')$ related through an arrow $ c' \rightarrow c''$ satisfying the commutation 
\[\begin{tikzcd}
	& {c'} \\
	d && c \\
	& {c''}
	\arrow["{u'}", from=1-2, to=2-3]
	\arrow["u"{description}, from=1-2, to=3-2]
	\arrow["{x'}", squiggly, from=2-1, to=1-2]
	\arrow["{x''}"', squiggly, from=2-1, to=3-2]
	\arrow["{u''}"', from=3-2, to=2-3]
\end{tikzcd}\]


At this point, $\lext_H(S)\rightarrow \widehat{H}(c)$ needs not be a monomorphism (although it is under assumption of flatness), so in full generality we must consider its image factorization 
\[\begin{tikzcd}
	{\lext_H(S)} && {\widehat{H}(c)} \\
	& {H[S]}
	\arrow["{\lext_H(m)}", from=1-1, to=1-3]
	\arrow[tail, from=1-1, to=2-2]
	\arrow[two heads, from=2-2, to=1-3]
\end{tikzcd}\]

For $d$ in $\D$, the presheaf $ H[S]$ returns the set of all $ x \in H(d,c)$ such that there exists a map $ u' : c' \rightarrow c$ and an heteromorphism $ x' \in H(d,c')$ such that $ H(d,u') (x') = x$, or in diagram
\[\begin{tikzcd}
	& {c'} \\
	{d} & c
	\arrow["{u'}", from=1-2, to=2-2]
	\arrow["{x'}", squiggly, from=2-1, to=1-2]
	\arrow["x"', squiggly, from=2-1, to=2-2]
\end{tikzcd}\]

Hence, though the lextension of $H$ does not return properly speaking a sieve over some object, it induces a subobject $ H[S]$ of the distributor itself, and at each heteromorphism, it looks locally as a sieve: indeed, by Yoneda lemma, an heteromorphism $ x \in H(d,c)$ is a transformation $x : \hirayo_d \rightarrow \widehat{H}(c) $, along which the pullback of the extension produces a sieve on $d$
\[\begin{tikzcd}
	{x^*H[S]} & {H[S]} \\
	{\hirayo_d} & {\widehat{H}(c)}
	\arrow[from=1-1, to=1-2]
	\arrow[tail, from=1-1, to=2-1]
	\arrow["\lrcorner"{anchor=center, pos=0.125}, draw=none, from=1-1, to=2-2]
	\arrow[tail, from=1-2, to=2-2]
	\arrow["x"', from=2-1, to=2-2]
\end{tikzcd}\]

Concretely, this object $ x^*H[S]$ returns at an object $ d'$ of $ \mathcal{D}$ the set of all morphisms $ v: d' \rightarrow d$ such that there is a morphism $ u : c' \rightarrow c$ in $S$ for some $c$ together with an heteromorphism $ x' : d' \rightarrow c'$ such that one has the equality in $H(d',c)$ 
\[ H(v,c)(x) = H(d',u)(x') \]
which visualizes in diagram as a factorization 
\[\begin{tikzcd}
	{d'} & {c'} \\
	d & c
	\arrow["{x'}", squiggly, from=1-1, to=1-2]
	\arrow["v"', from=1-1, to=2-1]
	\arrow["{u'}", from=1-2, to=2-2]
	\arrow["x"', squiggly, from=2-1, to=2-2]
\end{tikzcd}\]

\end{division}

This invites us to give the following definition:

\begin{definition}\label{cover-distribution}
A distributor $ H : \mathcal{C} \looparrowright \mathcal{D}$ between sites $ (\mathcal{C}, J)$ and $(\mathcal{D},K)$ will be said to be \emph{cover-distributing from $J$ to $K$} if for any pair $(d,c)$, any heteromorphism $ x \in H(d,c)$ and any $J$-covering sieve $ S$ on $c$, the sieve $ x^*H[S]$ is $K$-covering on $d$.
\end{definition}

\begin{lemma}\label{composition of cover dist distributors}
    Let be $H :  (\C,J) \looparrowright (\D,K)$ and $ G : (\D,K) \looparrowright (\B,L)$ distributors that are respectively $ (J,K)$-cover-distributing and $(K,L)$-cover-distributing; then the composite $ G \otimes H$ is $ (J,L)$-cover-distributing. Moreover for any site $ (\C,J)$ the unit $ \hom_\C$ is $(J,J)$-cover-distributing. 
\end{lemma}

\begin{proof}
    Let be $H,G$ as above; recall that $G\otimes H(b,c)$ is a coend of products of the form $ H(d,c) \times G(b,d)$, and we must show that for any element $ [(x,y)]_\sim$ of this coend, the corresponding pullback $[(x,y)]_\sim^*(G \otimes H)[S] $ is in $L(d)$. For a pair $ x \in H(d,c)$ and $y \in G(b,d)$ representing such an element, one has for each $S \in J(c)$ that $ x^*H[S] $ is in $K(d)$, and hence that $ y^*G[x^*H[S]]$ is in $L(b)$; this set is the set of all arrows $ w: b' \rightarrow b$ such that there is an arrow $ v: d' \rightarrow d$ and subsequently an arrow $ u : c' \rightarrow c$ in $S$ together with heteromorphisms as below
\[\begin{tikzcd}
	{b'} & {d'} & {c'} \\
	b & d & c
	\arrow["{y'}", squiggly, from=1-1, to=1-2]
	\arrow["w"', from=1-1, to=2-1]
	\arrow["{x'}", squiggly, from=1-2, to=1-3]
	\arrow["v"{description}, from=1-2, to=2-2]
	\arrow["u", from=1-3, to=2-3]
	\arrow["y"', squiggly, from=2-1, to=2-2]
	\arrow["x"', squiggly, from=2-2, to=2-3]
\end{tikzcd}\]
    Hence any $w : b' \rightarrow b$ in $y^*G[x^*H[S]]$ provides in particular an element of $[(x,y)]_\sim^*(G \otimes H)[S]$ after forgetting which representing element $(x,y)$ was chosen, and hence the sieve $[(x,y)]_\sim^*(G \otimes H)[S]$ is $L$-covering for it contains the sieve $y^*G[x^*H[S]]$ which is itself $L$-covering. Finally cover-distributy of $\hom_\C$ from $J$ to $J$ exactly expresses that $ J$ is pullback stable. 
\end{proof}

\begin{proposition}\label{cover dist subsume cover pres and cover lift}
    Let $ (\mathcal{C}, J)$ and $(\mathcal{D},K)$ be two sites; then: \begin{itemize}
        \item a functor $f : \mathcal{C} \rightarrow \mathcal{D}$ is cover-preserving if and only if $\mathcal{D}(1,f) : \mathcal{C} \looparrowright \mathcal{D}$ is cover-distributing;
        \item a functor $f : \mathcal{D} \rightarrow \mathcal{C}$ is cover-reflecting if and only if $\mathcal{C}(f,1) : \mathcal{C} \looparrowright \mathcal{D}$ is cover-distributing.
    \end{itemize}
\end{proposition}

\begin{proof}
 Both items rely on the formulas given at \cref{Extensions of representable distributors}.
 For the first item, take $f : \mathcal{C} \rightarrow \mathcal{D}$: one has for any sieve $ S$ in $J(c)$ in $\mathcal{C}$ that $\lext_{\mathcal{D}(1,f)}(S) = \lext_f(S)$, so they have the same image factorization, and hence for any $x : d \rightarrow f(c)$, corresponding to an heteromorphism $ x \in \mathcal{D}(1,f(c))$, the equality between the pullbacks. Hence, asking $f$ to be cover-preserving, which amounts by \cref{morphisms of sites in term of lext} to asking the pullback sieve $ x^*f[S]$ to be $K$-covering, amounts in turn to ask $ \mathcal{D}(1,f)$ to be cover-distributing.

 For the second item, take $f : \mathcal{D} \rightarrow \mathcal{C}$: for any sieve $ S$ in $J(c)$ in $\mathcal{C}$ and any $ x : f(d) \rightarrow c$ in $\mathcal{C}(f,c)$ the pullback sieve $ x^*S$ is in $J(f(d))$, and asking $f$ to be a comorphism of sites amounts by \cref{comorphism of sites in term of rest} to asking $f^{-1}(x^*S)$ to be $K$-covering for each $x$; but recall that for any object $ d'$ in $\mathcal{D}$, the set $ f^{-1}(x^*S)(d')$ consists exactly of those arrows $ v : d' \rightarrow d$ in $\mathcal{D}$ such that there is $u : c' \rightarrow c$ in $S$ and a factorization 
\[\begin{tikzcd}
	{f(d')} & {c'} \\
	{f(d)} & c
	\arrow["{x'}", from=1-1, to=1-2]
	\arrow["{f(v)}"', from=1-1, to=2-1]
	\arrow["{u'}", from=1-2, to=2-2]
	\arrow["x"', from=2-1, to=2-2]
\end{tikzcd}\]
But those are exactly the object of the pullback sieve evaluated at $d'$
\[\begin{tikzcd}
	{x^*H[S](d')} & {H[S](d')} \\
	{\mathcal{D}(d',d)} & {\mathcal{C}(f(d),c)}
	\arrow[from=1-1, to=1-2]
	\arrow[tail, from=1-1, to=2-1]
	\arrow["\lrcorner"{anchor=center, pos=0.125}, draw=none, from=1-1, to=2-2]
	\arrow[tail, from=1-2, to=2-2]
	\arrow["x"', from=2-1, to=2-2]
\end{tikzcd}\]

More conceptually, recall from \cref{Extensions of representable distributors} that we have a natural isomorphism $ \lext_{\C(f,1)} \simeq \rest_f$. Hence $ \lext_{\C(f,1)} $ is lex, and from \cref{comorphism of sites in term of rest}, $ f^{-1}(x^*S)$ is the fiber of $ \rest_f(x^*S)$, so we have an isomorphism $ f^{-1}(x^*S) \simeq x^*\lext_{\C(f,1)}(S) \simeq x^*\C(f,1)[S] $ where the last isomorphism comes from the fact that $\lext_{\C(f,1)}$ preserves mono since $ \rest_f$ does. Hence asking $f$ to be cover-reflecting amounts to asking $\mathcal{C}(f,1)$ to be cover-distributing.
\end{proof}

\begin{lemma}\label{cover dist implies send cover to epifam}
    Let $ H : \C \looparrowright \D$ a distributor between sites $(\C,J)$, $ (\D,K)$; then $H$ is cover-distributing if and only if $\mathfrak{a}_K\widehat{H}$ sends $J$-covers to epimorphic families in $ \widehat{\D}_K$.  
\end{lemma}

\begin{proof}
    Let be $ (u_i :c_i \rightarrow c)_{i \in I}$ a cover generating a $J$-covering sieve $S$. We must prove the induced map $\coprod_{i \in I} \widehat{H}(c_i) \rightarrow \widehat{H}(c)$ to be a $K$-local epimorphism in $\widehat{\D}$. This condition amounts to asking that for any $d$ in $\D$ and any $x \in H(d,c)$, there is a $K$-cover $(v_j: d_j \rightarrow d)_{j \in J}$ such that for any $ j \in J$, there is an antecedent $ x_j \in H(d_j,c_i)$ for some $i \in I$: but this condition exactly means that the sieve $R$ generated form $(v_j: d_j \rightarrow d)_{j \in J}$ is included in $x^*H[S]$: whence the equivalence between the two conditions. 
\end{proof}



Cover-distributivity plays the same role as cover-preservation does in the definition of morphisms of sites; but the latter involves also a second condition of \emph{cover-flatness} which is a generalization of flatness for functor valuated in sheaf topoi; it is a form of representable flatness up to epimorphic family. As much as flatness for functor valuated in presheaves generalized for distributor, we can introduce the following notion:

\begin{definition}\label{cover-flat distributor}
   Let $\C$ be a (small) category and $ (\D,K)$ a site; a distributor $ H : \C \looparrowright \mathcal{D}$ will be said to be \emph{$K$-flat} if:
	\begin{enumerate}

		\item For any $d$ of $\cal D$, the following sieve is $K$-covering 
  \[ {\{}v:d'\to d \mid \exists c\in {\cal C}, H(d',c)\neq \emptyset {\}} \] 
		
		\item For any $c, c'$, $d$ and $x \in H(d,c)$, $x' \in H(d,c')$, the following sieve is $K$-covering
		\[ \Bigg{\{}v:d'\to d \mid \begin{array}{l}
		    \exists u:c''\to c, \\ \exists u':c''\to  c' \\ \exists x'' \in H(d', c'')
		\end{array}  \textup{ such that } 	 
\begin{tikzcd}[row sep=small]
	{d'} & {c''} && c \\
	d &&& {c'}
	\arrow["{x''}", squiggly, from=1-1, to=1-2]
	\arrow["v"', from=1-1, to=2-1]
	\arrow["u", from=1-2, to=1-4]
    \arrow["x"{pos=0.2}, squiggly, from=2-1, to=1-4]
	\arrow["{u'}"{pos=0.7}, from=1-2, to=2-4]	
	\arrow["{x'}"', squiggly, from=2-1, to=2-4]
\end{tikzcd}  \Bigg{\}}  \]

		\item For any $u, u': c' \rightrightarrows c$ in $\cal C$, $d$ of $\cal D$ and $x \in H(d, c')$ with $H(d,u)(x)=H(d, u')(x)$, the following sieve is $K$-covering:	
	\[ 
		\Bigg{\{}v:d'\to d \mid \begin{array}{l}
		    \exists w : c'' \rightarrow c'  \\
		    \exists x' \in H(d', c'')
		\end{array} \textup{ such that } 
\begin{tikzcd}[row sep=small]
	{d'} & {c''} \\
	d & {c'} & c
	\arrow["{x'}", squiggly, from=1-1, to=1-2]
	\arrow["v"', from=1-1, to=2-1]
	\arrow["w"', from=1-2, to=2-2]
	\arrow["{uw= u'w}", from=1-2, to=2-3]
	\arrow["x"', squiggly, from=2-1, to=2-2]
	\arrow["u", shift left, from=2-2, to=2-3]
	\arrow["{u'}"', shift right, from=2-2, to=2-3]
\end{tikzcd} \Bigg{\}}
\]		
	\end{enumerate}
\end{definition}

\begin{remark}\label{cover-flatness for distributor synthetic form}
    We could synthetisize this condition into a more direct analog of cover-flatness as defined in \cite{shulman2012exact} and \cite{karazeris2004notions}: we could say that $H : \C \looparrowright \D$ is $K$-flat if for any finite diagram $(c_i)_{i \in I}$ in $\C$, any $d $ in $\D$ and any $ (x_i  \in H(d,c_i))_{i \in I}$ a compatible family of heteromorphisms, the family 
    \[  \Big{\{} v : d' \rightarrow d \mid \textup{ there is a cone } (u_i : c \rightarrow c_i)_{i \in I} \textup{ and } x \in H(d',c)  \textup{ with } 
\begin{tikzcd}
	{d'} & c \\
	d & {c_i}
	\arrow["x", squiggly, from=1-1, to=1-2]
	\arrow["v"', from=1-1, to=2-1]
	\arrow["{u_i}", from=1-2, to=2-2]
	\arrow["{x_i}"', squiggly, from=2-1, to=2-2]
\end{tikzcd} \Big{\}} \]
is to be $K$-covering. Remark that, as for cover-flatness of functors, this condition does not involve any coverage on the domain category. 
\end{remark}

\begin{lemma}\label{K-flat distributor equivalences}
    For a distributor $H: \mathcal{C} \looparrowright \mathcal{D}$, the following conditions are equivalent:\begin{itemize}
        \item $H$ is a $K$-flat distributor;
        \item the induced functor $ \mathfrak{a}_K\widehat{H } :  \C \rightarrow \widehat{\mathcal{D}}_K$ is flat;
        \item the induced functor $\mathfrak{a}_K\lext_H \mathfrak{i}_J : \widehat{\C}_J \rightarrow \widehat{\D}_K$ is lex.
    \end{itemize}   
\end{lemma}

\begin{proof}
    The induced $ \widehat{H}$ sends $ c$ to the presheaf $ H(d)$; by \cref{cover dist implies send cover to epifam} we know that, since $H$ is in particular cover distributive, $\mathfrak{a}_K\widehat{H}$ sends $J$-cover to epimorphic families in $\widehat{\D}_K$; it remains to show that it is also flat. To do this we show that $ \mathfrak{a}_K \widehat{H}$ is filtering in the sense of \cite{maclane&moerdijk}[VII.8 definition 1]. In each case, one must show that some family involving $ \mathfrak{a}_K\widehat{H}$ is epimorphic in $\widehat{\D}_K$, but again this amounts to showing that the corresponding family defines a local epimorphism from a coproduct in $\widehat{\D}$, where the coproduct is computed pointwisely, and one recovers an equivalence between each of the three requirements of filteredness with the corresponding statement in our definition of flatness: \begin{itemize}
        \item asking the family $ \{ \widehat{H}(c) \rightarrow 1 \mid c \in \C \}$ to be locally epimorphic amounts to asking that at each $d$, the set of all $v : d' \rightarrow d$ such that the singleton as an antecedent in $H(d',c)$ for some $c$ in $\C$ is $K$-covering: this is exactly equivalent to our condition (1);
        \item for a pair of objects $c,c'$ in $\C$, asking the family $\{ (\widehat{H}(u),\widehat{H}(u')) :  \widehat{H}(c'') \rightarrow \widehat{H}(c) \times \widehat{H}(c') \mid  u: c'' \rightarrow c, u':c'' \rightarrow c' \}$ to be locally epimorphic amounts to asking that at each $d$, and any $(x,x') \in (\widehat{H}(c) \times \widehat{H}(c'))(d) \simeq H(d,c) \times H(d,c') $, the set of all $v :d' \rightarrow d$ such that there is some $x'' \in H(d',c'')$ for some span $u : c'' \rightarrow c, u' : c'' \rightarrow c'$ with $ H(d',u)(x'') = H(v,c)(x)$ and $H(d',u')(x'') = H(v,c')(x')$ is $K$-covering: this is exactly equivalent to our condition (2);
        \item for a pair or arrows $ u,u': c' \rightarrow c$ in $\C$, asking the family $ \{ \langle \widehat{H}(w) \rangle : \widehat{H}(c'') \rightarrow \eq(\widehat{H}(u), \widehat{H}(u')) \mid w : c'' \rightarrow c' \textup{ such that } uw=u'w  \}$ to be locally epimorphic amounts to asking that for any $x \in \eq(\widehat{H}(u), \widehat{H}(u'))(d)$, that is, for any $x \in H(d,c')$ such that $ H(d,u)(x) = H(d,u')(x)$ in $H(d,c)$, the set of all $v : d' \rightarrow d$ such that there is some $ w : c'' \rightarrow c'$ with $uw=u'w$ together with $ x' \in H(d', c'')$ with $ H(v,c')(x) = H(d',w)(x')$ is $K$-covering: this is exactly equivalent to our condition (3).
    \end{itemize}
    By \cite{maclane&moerdijk}[VII.9 theorem 1] the functor $ \mathfrak{a}_K\widehat{H}$ is filtering if and only if it is flat, and this amounts for the induced $ \mathfrak{a}_K \lext_H$ to being lex. 
\end{proof}

\subsection{Distributors of sites and continuity conditions}

Now, as in the functorial case, cover-distribution and cover-flatness can be mixed altogether do define an appropriate notion of distributors between sites:

\begin{definition}
    Let $(\C,J)$, $(\D,K)$ be sites; a distributor $ H : \C \rightarrow \D$ will be said to be a \emph{distributor of sites}, denoted as $ H : (\C,J) \looparrowright (\D,K)$, if it is cover-distributing from $J$ to $K$ and $K$-flat. 
\end{definition}

\begin{proposition}\label{distributor of sites induce geomorph}
    Let $ H : (\C,J) \looparrowright (\D,K)$ be a distributor of sites; then $ \mathfrak{a}_K\lext_H$ factorizes through $ \mathfrak{a}_J$ and defines the inverse image part of a geometric morphism $\Sh(H) : \widehat{\D}_K \rightarrow \widehat{\C}_J$ as depicted below
\[\begin{tikzcd}
	{\widehat{\mathcal{C}}} & {\widehat{\mathcal{D}}} \\
	{\widehat{\mathcal{C}}_J} & {\widehat{\mathcal{D}}_K}
	\arrow["{\lext_H}", from=1-1, to=1-2]
	\arrow["{\mathfrak{a}_J}"', from=1-1, to=2-1]
	\arrow["{\mathfrak{a}_K}", from=1-2, to=2-2]
	\arrow["{\Sh(H)^*}"', dashed, from=2-1, to=2-2]
\end{tikzcd}\]

\end{proposition}

\begin{proof}
    By $K$-flatness, $ \mathfrak{a}_K\widehat{H}$ is a flat functor $ \C \rightarrow \widehat{\D}_K$: hence its extension $ \lan_{\hirayo_{\C}} \mathfrak{a}_K\widehat{H}$ is lex, but as a left adjoint, $\mathfrak{a}_K$ commutes with left Kan extensions, so $\lan_{\hirayo_{\C}} \mathfrak{a}_K\widehat{H} \simeq \mathfrak{a}_K \lext_H$. But since $H$ is $K$-flat and cover-distributing from $J$ to $K$, then $ \mathfrak{a}_K\widehat{H} : \C \rightarrow \widehat{\D}_K$ is a $J$-flat functor, hence by Diaconescu, its extension $ \lan_{\hirayo_{\C}} \mathfrak{a}_K\widehat{H} $ factorizes through the sheafification $\mathfrak{a}_J$, in the sense that $ \mathfrak{a}_K \widehat{H}$ inverts the units of the adjunction $ \mathfrak{a}_J \dashv \mathfrak{i}_J$. 
\end{proof}

\begin{lemma}
    Let $H :  (\C,J) \looparrowright (\D,K)$ and $ G : (\D,K) \looparrowright (\B,L)$ be distributors of sites; then the composite $ G \otimes H $ defines a distributor of sites $ (\C,J) \looparrowright (\D,K)$. Moreover for any site $ (\C,J)$ the unit $ \hom_\C$ is a distributor of sites $(\C,J) \looparrowright (\C,J)$. 
\end{lemma}

\begin{proof}
    We saw at \cref{composition of cover dist distributors} that cover-distributing distributors compose and that the unit is cover-distributing. Now suppose that $H$, $G$ are distributor of sites; recall that 
    \[  \widehat{G \otimes H} \simeq \lext_G \widehat{H} \]
    But also recall that $\mathfrak{a}_L\lext_G$ factorizes through $ \mathfrak{a}_K$ for $G$ is a distributor of sites, hence $ \mathfrak{a}_L \lext_G $ inverts the units of the adjunction $ \mathfrak{a}_K \dashv \mathfrak{i}_K$ and we have hence 
    \[ \mathfrak{a}_L \lext_G \widehat{H} \simeq  
 \mathfrak{a}_L \lext_G \mathfrak{i}_K \mathfrak{a}_K \widehat{H}  \]
    but on one hand, $\mathfrak{a}_L \lext_G \mathfrak{i}_K $ is a lex functor since $G$ is $L$-flat, and on the other hand, $ \mathfrak{a}_K \widehat{H}$ is flat since $H$ is $K$-flat, so this composite is flat, and from \cref{K-flat distributor equivalences}, the composite $ G \otimes H$ must be $L$-flat; being also cover-distributing, it is hence a distributor of sites. Regarding the unit $\hom_{\C}$, we have seen it is cover-distributing; but its classifier is $ \hirayo_{\C} : \C \hookrightarrow \widehat{\C}$, and the composite $ \mathfrak{a}_J \hirayo_{\C}$ is always flat as it induces the identity functor of $ \widehat{\C}_J$. 
\end{proof}

Being closed under composition and unit, we can consider the bicategory $ \Site^{\mustrill}$ sites, distributors of sites and transformations.

\begin{theorem}\label{functoriality for distributor of sites}
    The construction above defines a pseudofunctor 
\[\begin{tikzcd}
	{(\Site^{\mustrill})^{\op}} & \Top
	\arrow["\Sh", from=1-1, to=1-2]
\end{tikzcd}\]
\end{theorem}

\begin{proof}
    We have to show this construction is compatible with composition of distributors of sites and extends to transformations between them. Taking a pair of distributors of sites $ H,G$ as in the previous lemma, we have that \begin{align*}
        \Sh(G \otimes H)^* &\simeq \mathfrak{a}_L \lext_{G \otimes H} \mathfrak{i}_J \\ 
        &\simeq \mathfrak{a}_L \lext_G \lext_H \mathfrak{i}_J \\
        &\simeq \mathfrak{a}_L \lext_G \mathfrak{i}_K \mathfrak{a}_K \lext_H \mathfrak{i}_J \\
        &\simeq \Sh(G)^*\Sh(H)^*
    \end{align*} 
    where the third isomorphism comes from the fact that $ \mathfrak{a}_L\lext_G$ inverts $K$-bidense morphisms. Similarly we saw that
    \begin{align*}
        \Sh(\hom_\C)^* &\simeq \mathfrak{a}_J \lan_{\hirayo_\C} \hirayo_\C \\
        &\simeq 1_{\widehat{\C}_J}
    \end{align*}
    Regarding transformations of distributors, for $ \phi : H \Rightarrow G$, the functoriality of the equivalence $ \Dist[\C,\D] \simeq [\C, \widehat{\D}_K]$ induces a natural transformation $ \widehat{\phi } : \widehat{H} \Rightarrow \widehat{G} $, which in turn induces in a covariant way a transformation $ \phi^\flat : \lext_H \Rightarrow \lext_G  $ from which we deduce by sheafification the inverse image part of a geometric transformation $ \Sh(\phi)$. Hence pseudofunctoriality of $\Sh$. 
\end{proof}

In \cite{johnstonewraith}[Chapter III] is also considered a property for distributors that is sufficient to induce geometric morphisms, which is in fact almost tautological, in terms of {bidense morphisms}. Recall that a Grothendieck topology $J$ induces a lex-localization $ \mathfrak{a}_J : \widehat{\mathcal{C}}^{\op} \rightarrow \Sh(\mathcal{C}, J)$.

\begin{definition}
    A \emph{dense monomorphism} in $[\widehat{\mathcal{C}}^{\op}, \Set]$ is a monomorphism $m : M \rightarrowtail E$ whose image $ \mathfrak{a}_J(m) $ is an isomorphism; more generally a \emph{dense morphism} is a morphism $u: E' \rightarrow E$ whose image part $ m_u $ is a dense monomorphism; finally a \emph{bidense morphism} is a dense morphism whose diagonal $ \Delta_u : E' \rightarrowtail E'\times_E^{u,u} E'$ is a dense monomorphism. 
\end{definition}

\begin{proposition}
For a Grothendieck topology $ J$ on $\mathcal{C}$, the sheaf topos $ \widehat{\mathcal{C}}_J$ is exhibited as the localization of $ \widehat{\mathcal{C}}$ at $J$-bidense morphisms. A geometric morphism $ f : \widehat{\mathcal{C}} \rightarrow \widehat{\mathcal{D}}$ restricts to a geometric morphism $ \widehat{\mathcal{C}}_J \rightarrow \widehat{\mathcal{D}}_K$ if and only if it sends $J$-dense monomorphisms to $K$-denses monomorphisms (equivalently, if it sends $J$-bidense morphisms to $K$-bidense morphisms).
\end{proposition}

\begin{definition}
    A distributor $ \mathcal{C} \looparrowright \mathcal{D}$ is \emph{Johnstone-Wraith-continuous} if, for any finite diagram $ (E_i)_{i \in I} : I \rightarrow \widehat{\mathcal{C}}$ and any $J$-bidense morphism $ u : E \rightarrow \lim_{i \in I} E_i$, the induced map $\lext_H(u) : \lext_H(E) \rightarrow \lext_H(\lim_{i \in I}E_i)$ is $K$-bidense in $\widehat{\mathcal{D}}$. 
\end{definition}

\begin{remark}
    Observe that in particular, for this applies to the case of constant diagrams $ E$ (which are finite), this condition implies that $\lext_H$ sends $J$-bidense morphisms to $K$-bidense morphisms.
\end{remark}

For an ordinary coverage $ J$ on $\C$ which is not required to satisfy the transitivity axiom, dense monomorphisms $ m : S \rightarrowtail \hirayo_c$ corresponds to sieves that are covering for the Grothendieck topology generated from $J$ (by closing under the transitivity). More generally, a dense monomorphism into an arbitrary object is a monomorphism whose pullback along representables are dense, see \cite{borceux1994handbook3}[proposition 3.5.2].

\begin{proposition}
    A distributor $ H : \C \looparrowright \D$ defines a distributor of sites $(\C,J) \looparrowright (\D, \overline{K})$ (for the transitive closure $\overline{K}$ of $K$) if and only if it is Johnstone-Wraith continuous for $J,K$. 
\end{proposition}

\begin{proof}
    Suppose $H$ to define a distributor of sites $(\C,J) \looparrowright (\D, \overline{K})$; as $K$ and $\overline{K} $ induce the same sheaf topos, we have $ \mathfrak{a}_K \simeq \mathfrak{a}_{\overline{K}}$, then $\mathfrak{a}_K\lext_H$ factorizes through $ \mathfrak{a}_J$, hence localizes $J$-bidenses morphisms, which exactly means that it sends $J$-bidense morphisms to $\overline{K}$-bidense morphisms, which are the same as the $K$-bidense morphisms. 

    Conversely suppose that $H$ is Johnstone-Wraith continuous for $J,K$; let be $ S \rightarrow \hirayo_c$ a $J$-covering sieve: this is a $J$-bidense morphism, hence is sent by $\lext_H$ to a $K$-bidense morphism $ \lext_H(S) \rightarrow \widehat{H}(c)$, so that its image factorization $ H[S] \rightarrowtail \widehat{H}(c)$ is a dense monomorphism, and for any $x \in H(d,c)$, $x^*H[S] \rightarrow \hirayo_d$ is a dense monomorphism in a representable, hence corresponds to a $\overline{K}$-covering sieve on $d$. Moreover, applying this fact to bidense morphism into finite limits, $\mathfrak{a}_K\lext_H$ is exhibited as being lex, hence $H$ to be $K$-flat.     
\end{proof}

\begin{remark}
   If $H$ defines a distributor of sites $(\C,J) \looparrowright (\D, K)$, then it also defines a distributor of sites $(\C,J) \looparrowright (\D, \overline{K})$ as any $K$-covers provides a $\overline{K}$-cover.  
\end{remark}

We can then give the following variant of the proof of \cite{johnstonewraith}[III, proposition 3.2]:

\begin{proposition}\label{adjunction}
    Let $ (\mathcal{C},J)$ and $(\mathcal{D},K)$ be small generated sites; there is an adjunction 
\[\begin{tikzcd}
	{\Top[\widehat{\mathcal{D}}_K, \widehat{\mathcal{C}}_J]} && {\Site^{\mustrill}((\mathcal{C},J)
(\mathcal{D},K))}
	\arrow[""{name=0, anchor=center, inner sep=0}, "{\mathbb{H}}"', curve={height=24pt}, from=1-1, to=1-3]
	\arrow[""{name=1, anchor=center, inner sep=0}, "{\Sh}"', curve={height=24pt}, from=1-3, to=1-1]
	\arrow["\dashv"{anchor=center, rotate=-90}, draw=none, from=1, to=0]
\end{tikzcd}\]
where moreover the functor $\mathbb{H}$ is fully faithful.
\end{proposition}

\begin{proof}
We have seen at \cref{functoriality for distributor of sites} that $\Sh$ is part of a pseudofunctor $ (\Site^{\mustrill})^{\op} \rightarrow \Top$ sending $H$ to the geometric morphism $ \Sh(H)$ with inverse image $ \mathfrak{a}_K\lext_H \mathfrak{i}_J$. Conversely, if $ f : \widehat{\D}_K \rightarrow \widehat{\C}_J$ is a geometric morphism, we can extract a distributor $ H_f : \C \looparrowright \D$ sending $ (c,d)$ to the evaluation $ f^*(\mathfrak{a}_J\hirayo_c)(d)$, seeing $f^*(\mathfrak{a}_J\hirayo_c) $ as a sheaf on $(\D,K)$. We must prove that $H_f$ is both $K$-flat and cover distributing from $J$ to $K$, but in both cases, this is because the functor $ \mathfrak{a}_K\widehat{H_f}$ is so, for one has $ \widehat{H_f} : \C \rightarrow \widehat{\D}$ coincides with the functor $ \mathfrak{i}_Kf^*\mathfrak{a}_J\hirayo_{\C} $ as depicted below
\[\begin{tikzcd}
	\C & {\widehat{\D}} \\
	{\widehat{\C}_J} & {\widehat{\D}_K}
	\arrow["{\widehat{H_f}}", from=1-1, to=1-2]
	\arrow["{\mathfrak{a}_J \hirayo_\C}"', from=1-1, to=2-1]
	\arrow["{f^*}"', from=2-1, to=2-2]
	\arrow["{\mathfrak{i}_K}"', hook, from=2-2, to=1-2]
\end{tikzcd}\]
and this functor is flat and sends $J$-cover to epimorphic families, as predicted by Diaconescu correspondence; but from \cref{K-flat distributor equivalences} and \cref{cover dist implies send cover to epifam}, those two conditions translate for $H_f$ as being $K$-flat and cover-distributing from $J$ to $K$. Observe that, starting form a geometric morphism $f$, we have $ \mathfrak{a}_K \widehat{H_f} \simeq f^* \mathfrak{a}_J \hirayo_\C $ 
\begin{align*}
    \Sh(H_f)^* &\simeq \mathfrak{a}_K\lext_{H_f} \mathfrak{i}_J \\ &\simeq \lan_{\hirayo_\C} (\mathfrak{a}_K\widehat{H_f}) \mathfrak{i}_J \\
    &\simeq \lan_{\hirayo_\C} (f^* \mathfrak{a}_J \hirayo_\C) \mathfrak{i}_J \\
    &\simeq f^*
\end{align*}
where the last isomorphism comes from the fact that $f^*$ is the inverse image of a geometric morphism and is induced from its restriction $ \C \rightarrow \widehat{\D}_K$ through Diaconescu equivalence (beware that the $\lan$ above are computed in sheaves). However, on the other hand, taking a distributor $ H$ and then $ H_{\Sh(H)}$ only recovers the distributor $ \C \looparrowright \D$ sending $ (d,c)$ to $ \mathfrak{a}_K\widehat{H}(c)(d)$, which may be distinct from $ H(d,c)$ itself. 

We must now establish the desired adjunction: using that $ \Site^{\mustrill}[(\C,J), (\D,K)]$ is a full subcategory of $ \Dist[\C, \D]$ and that similarly $ \Cont_J[\C, \widehat{\D}_K]$ is a full subcategory of $ [\C,\widehat{\D}_K]$, we have a sequence of isomorphisms
\begin{align*}
    \Site^{\mustrill}[(\C,J), (\D,K)] &\simeq \Dist[\C, \D][H,H_f] \\ &\simeq [\C, \widehat{\D}][\widehat{H}, \widehat{H_f}] \\
    &\simeq [\C, \widehat{\D}][\widehat{H}, \mathfrak{i}_K f^*\mathfrak{a}_J\hirayo_\C] \\
    &\simeq \Cont_J[\C, \widehat{\D}_K][\mathfrak{a}_K\widehat{H}, f^*\mathfrak{a}_J\hirayo_\C] \\
    &\simeq \Top[\widehat{\mathcal{D}}_K, \widehat{\mathcal{C}}_J][\Sh(H), f]
\end{align*}
\end{proof}

The hindrance to an equivalence is the problem that sheafification $ \mathfrak{a}_K$ may identify distributors. This can be fixed by adding a condition of continuity:

    \begin{definition}\label{defrelativeprofunctor}
    A distributor of sites $ H : (\mathcal{C}, J) \looparrowright (\mathcal{D},K)$ will be said to be \emph{$(J,K)$-continuous} if it satisfies that $\widehat{H}:{\cal C}\to \widehat{\mathcal{D}}$ takes values in the sheaf topos $\widehat{\mathcal{D}}_K$. 

\end{definition}

\begin{remark}
    The condition says that for each $c$ in $\mathcal{C}$, $ \widehat{H}(c)$ is a $K$-sheaf: for any $d$ and any $K$-covering sieve $T$ of $d$, one has a limit
    \[ H(d,c) \simeq \underset{v : d' \rightarrow d \in T}{\lim} \; H(d',c) \]
\end{remark}

\begin{remark}
    For any distributor of sites $ H : (\C,J) \looparrowright (\D, K)$, it is true by construction that $ \mathfrak{a}_K\widehat{H}$ lands in $\widehat{\D}_K$, but this process may identify distinct distributors that are not isomorphic nor even directly related by an arrow. However, if we assume that $ \widehat{H}$ lands in $\widehat{\D}_K$, then one can compute the left extension $ \lan_{\hirayo_\C} \widehat{H}$ directly in $\widehat{\D}_K$, and Diaconescu equivalence ensures that $\Sh(H)^*$ is uniquely determined by $\widehat{H}$ itself.   
\end{remark}


\begin{remark}
    Beware that continuous distributors of sites do not compose inside of $\Site^{\mustrill}$: for two continuous distributors of sites $ H : (\C,J) \looparrowright (\D, K)$ and $ G : (\D,K) \looparrowright (\B, L)$, the composite $ G \otimes 
    H$ may not be continuous; although $\widehat{G}$ lands in $\widehat{\B}_L$, its extension $ \lan_{\hirayo_\D} \widehat{G}$, though factorizing through $ \widehat{\D}_K$, needs not again landing in $\widehat{\B}_L$ for sheaves are not closed under colimits in $ \widehat{\D}$ (one has to perform sheafification after computing colimits as presheaves). For this reason the classifying functor of the composite $ \widehat{G \otimes H}$, which is $ \lext_G \widehat{H}$, needs not landing in $\widehat{\B}_L$, and has to be sheafified.
    In a similar way, the unit distributors $ \hom_\C : (\C, J) \looparrowright (\C,J)$ needs not be continuous, as its classifying functor is $ \hirayo_\C : \C \hookrightarrow \widehat{\C}$ which does not always land in $\widehat{\C}_J$ -- this happens exactly when $J$ is subcanonical. 
\end{remark}

\begin{division}
    Consequently, to define a bicategorical structure for continuous comorphisms, one needs to modify the notion of unit and composition:\begin{itemize}
        \item for a site $(\C,J)$, define its unit as the distributor $ \mathfrak{h}_{(\C,J)} : (\C,J) \looparrowright (\C,J)$ sending $ (c,c')$ to the sheafification of the homset $ \mathfrak{a}_J\hirayo_c(c')$, to that one has 
        \[ \widehat{\mathfrak{h}_{(\C,J)}} \simeq \mathfrak{a}_J \hirayo_\C \]
        \item for two continuous distributors of sites $ H : (\C,J) \looparrowright (\D, K)$ and $ G : (\D,K) \looparrowright (\B, L)$, define the composite $ G \boxtimes H$ as the continuous distributor whose classifying functor is the sheafification of $ G\otimes H$ as distributors of sites
        \[ \widehat{G \boxtimes H} \simeq \mathfrak{a}_L\lext_G \widehat{H}  \]
    \end{itemize}
    Then one can check that $ \mathfrak{h}_{(\C,J)}$ is a unit for this composition which is weakly associative: this defines a structure of bicategory for (small generated) sites, continuous distributors of sites and transformations between them, which we will denote as $ \Site^{\mustrill}_{\Cont}$ this bicategory.    
\end{division}

\begin{remark}
    For they have different composition and units, beware that $\Site^{\mustrill}_{\Cont}$ is \emph{not} a sub-2-category of $\Site^{\mustrill}$: it rather is a 2-localization obtained by inverting formally a class of 2-cells corresponding to pointwise sheafifications, so we have a pseudofunctor 
    \[\begin{tikzcd}
	   {\Site^{\mustrill}} & {\Site^{\mustrill}_{\Cont}}
	   \arrow["{\mathfrak{a}}", from=1-1, to=1-2]
    \end{tikzcd}\]
    sending $ H : (\C,J) \rightarrow  (\D,K)$ to the pointwise sheafification $ \mathfrak{a}H : (\C,J) \rightarrow  (\D,K)$ whose classifying functor is defined at each $c$ in $\C$ as
    \[ \widehat{\mathfrak{a}H} = \mathfrak{i}_K\mathfrak{a}_K\widehat{H}(c) \]
    This pseudofunctor is bijective on objects and surjective on 1-cells, but it inverts transformations of distributors $ \phi :  H \Rightarrow H'$ such that $ \mathfrak{a}_K \widehat{\phi} : \mathfrak{a}_K\widehat{H} \Rightarrow \mathfrak{a}_K\widehat{H'}$ is invertible. In particular it inverts the transformations classified by pointwise sheafifications $ \widehat{H} \Rightarrow \mathfrak{a}_K \widehat{H}$.

\end{remark}

\begin{proposition}
    We have a factorization
\[\begin{tikzcd}
	{(\Site^{\mustrill})^{\op}} && \Top \\
	& {(\Site^{\mustrill}_{\Cont})^{\op}}
	\arrow["\Sh", from=1-1, to=1-3]
	\arrow["{\mathfrak{a}^{\op}}"', two heads, from=1-1, to=2-2]
	\arrow["\overline{\Sh}"', dashed, from=2-2, to=1-3]
\end{tikzcd}\]
    where the left hand of the factorization is locally fully faithful. 
\end{proposition}

\begin{proof}
    We have to define first a pseudofunctor from $\Site^{\mustrill}_{\Cont}$, which has to be compatible with its own bicategorical structure. But by construction $ \Sh(H)^* = \mathfrak{a}_K \lext_H \mathfrak{i}_J $; this is the extension of $ \mathfrak{a}_K\widehat{H}$. Hence for $ H : (\C,J) \looparrowright (\D, K)$ and $ G : (\D,K) \looparrowright (\B, L)$ two continuous distributors of sites, we have
    \begin{align*}
         \Sh(G \otimes H)^* &= \mathfrak{a}_L \lext_{G \otimes H} \mathfrak{i}_J \\ 
         &\simeq \mathfrak{a}_L \lext_{G }\lext_{H} \mathfrak{i}_J \\
        &\simeq \mathfrak{a}_L (\lext_{G } \lan_{\hirayo_\C} \widehat{H} ) \mathfrak{i}_J \\
        &\simeq \lan_{\hirayo_\C}(\mathfrak{a}_L \lext_{G }  \widehat{H}) \mathfrak{i}_J \\
        &\simeq \lan_{\hirayo_\C} \widehat{G \boxtimes H} \mathfrak{i}_J \\
        &\simeq \mathfrak{a}_L\lext_{G \boxtimes H} \mathfrak{i}_J \\
        &\simeq \Sh(G \boxtimes H)^*
    \end{align*}
    Similarly, one has 
    \begin{align*}
        \Sh(\hom_\C)^* &\simeq \mathfrak{a}_J \lan_{\hirayo_\C}\hirayo_\C \mathfrak{i}_J \\ 
        &\simeq \lan_{\hirayo_\C}\mathfrak{a}_J \hirayo_\C \mathfrak{i}_J \\
        &\simeq \Sh(\mathfrak{h}_{(\C,J)})^*
    \end{align*}
    Now if $ \phi : H \Rightarrow H'$ is a transformation of distributors of sites (which may not be continuous) such that $ \mathfrak{a}\phi$ is invertible; this means that $ \mathfrak{a}_K\phi$ is invertible, to that as, being isomorphic, hence $ \Sh(\phi)^* : \Sh(H)^* \simeq \Sh(H')^*$: so $\Sh$ localizes the same transformations of distributors as $\mathfrak{a}$, so it factorizes through it.   
\end{proof}

In fact, for $ (\C,J), (\D,K)$, the category $\Site^{\mustrill}_{\Cont}[(\mathcal{C},J), (\mathcal{D},K)] $ is a reflective subcategory of $ \Site^{\mustrill}[(\mathcal{C},J), (\mathcal{D},K)]$, and in fact precisely the one for which the adjunction of \cref{adjunction} restricts to an equivalence: 

\begin{theorem}\label{continuous distributors of sites corresponds with geometric morphisms}
    For any small generated sites $(\C,J)$, $(\D,K)$ we have an equivalence of categories
    \[  \Top[\widehat{\mathcal{D}}_K, \widehat{\mathcal{C}}_J] \simeq  \Site^{\mustrill}_{\Cont}[(\mathcal{C},J), (\mathcal{D},K)] \]
\end{theorem}

\begin{corollary}
    For two sites $(\mathcal{C},J)$, $(\mathcal{D},K)$, any geometric morphism $\widehat{\mathcal{D}}_K \rightarrow \widehat{\mathcal{C}}_J $ can be induced uniquely as $\Sh(H_f)$ from a continuous distributor of sites $ H_f : (\mathcal{C},J) \looparrowright  (\mathcal{D},K) $.
\end{corollary}

We saw at \cref{cover dist subsume cover pres and cover lift} that a functor between sites is cover-preserving (resp. cover-lifting) if and only if the corresponding representable (resp. corepresentable) is cover-distributing: we are now able to define embeddings of the 2-categories $ \Site^\flat$ and $\Site^\sharp$ into the double category $ \Site^{\mustrill}$ of sites, distributors of sites and transformations, where $ \Site^\flat$ and $\Site^\sharp$, using our terminology introduced in \cite{caramello2025morphisms}, refer to the 2-categories of sites with, respectively, morphisms and comorphisms as 1-cells, and natural transformations as 2-cells.

\begin{lemma}
    For any morphism of sites $ f : (\C,J) \rightarrow (\D,K)$, the corresponding representable $ \D[1, f] $ defines a distributor of sites $ (\C, J) \looparrowright (\D,K)$. 
\end{lemma}

\begin{proof}
    We saw that if $f$ is cover-preserving, then $ \D[1,f]$ is cover-distributing. Now suppose that $ f$ is also covering-flat, then $ \mathfrak{a}_K \hirayo_{\D} f : \C \rightarrow \widehat{\D}_K $ is a also flat functor, but recall that $\hirayo_{\D} f = \D[1,f]$ seen as a its own classifying functor $ \C \rightarrow \widehat{H}$: hence by \cref{K-flat distributor equivalences} we know that $ \D[1,f]$ is $K$-flat as a distributor. 
\end{proof}

\begin{lemma}
    For any comorphism of sites $ f : (\D,K) \rightarrow (\C,J)$, the corresponding corepresentable $ \C[f,1]$ is a distributor of sites $ (\C, J) \looparrowright (\D,K)$. 
\end{lemma}

\begin{proof}
    We saw at \cref{cover dist subsume cover pres and cover lift} that in this case $ \C[f,1]$ is cover-distirbuting from $J$ to $K$; but moreover, from \cite{maclane&moerdijk}[VII.10 theorem 5] we know that $ \mathfrak{a}_K\C[f,1]$ is also flat: hence by \cref{K-flat distributor equivalences} we know that $ \C[f,1]$ is $K$-flat as a distributor. 
\end{proof}

\begin{theorem}
The left and right representable construction produce factorizations of $\Sh : (\Site^{\flat})^{\op} \rightarrow \Top$ and $C : (\Site^\sharp)^{\co} \rightarrow \Top$ through the functor $ \Sh : \Site^{\mustrill} \rightarrow \Top$ as depicted below:
\[\begin{tikzcd}
	{(\Site^\flat)^{\op}} & {(\Site^{\mustrill})^{\op}} & {(\Site^\sharp)^{\co}} \\
	& \Top
	\arrow["R", from=1-1, to=1-2]
	\arrow["\Sh"', from=1-1, to=2-2]
	\arrow["{\widehat{(-)}}", from=1-2, to=2-2]
	\arrow["L"', from=1-3, to=1-2]
	\arrow["C", from=1-3, to=2-2]
\end{tikzcd}\]
\end{theorem}

\begin{proof}
    We have to show that the processes described in the two lemma above are pseudofunctorial relative to transformation between (co)morphisms of sites. A natural transformation $ \phi : f \Rightarrow g$ induces in a covariant way a transformation $ \D[1,\phi] : \D[1,f] \Rightarrow \D[1,g]$ between the corresponding distributors, from which we can recover the inverse image part of a geometric transformation $ \Sh(\phi)$; similarly, it induces in a contravariant way a transformation of distributors $ \D[\phi, 1] : \D[g,1] \Rightarrow \D[f,1]$, from which one can recover the inverse image part of a geometric transformation $C_\phi :  C_g \Rightarrow C_f$. 
\end{proof}

\begin{example}
    We discuss here an example from categorical logic. Consider two small pretopoi $ \C,\D $; then from \cite{MAKKAIduality} and \cite{UltraLurie} one has an equivalence of categories 
    \[  \preTop[\C,\D] \simeq \Coh\Top[\Sh(\D,J_\Coh ), \Sh(\C,J_\Coh ) ] \]
    where $ \Coh\Top$ is the category of coherent geometric morphisms, those $f : \Sh(\D,J_\Coh) \rightarrow \Sh(\C,J_\Coh)  $ whose inverse image $ f^*$ preserves coherent objects. However not all geometric morphism between the associated coherent topoi needs to be coherent, that is, to be induced by a morphism of pretopoi. Take an arbitrary geometric morphism $ f $ as above; then define the following distributor $ H_f : \C \looparrowright \D$ with $ H_f(d,c) = f^*(\hirayo_c)(d)$. Then $H_f$ is flat and cover-distributing relative to the coherent topologies. Equivalently this means that for any finite diagram $(c_i)_{i \in I}$ in $\C$, one has an isomorphism of in $ \widehat{\D}_K$
    \[ \widehat{H}_f(\underset{i \in I}{\colim} \, c_i) \simeq \underset{i \in I}{\colim}  \, \widehat{H}_f(c_i)  \]
    Let us call continuous ditributors with this property \emph{coherent distributors}. Then one has an equivalence of category 
    \[  \Coh\Dist[\C,\D] \simeq \Top[\Sh(\D,J_\Coh ), \Sh(\C,J_\Coh )]  \]
\end{example}

\begin{remark}
    This latter example is the categorified analog of \cite{AbramskyJung}[definition 7.2.24] where is introduced the notion of \emph{join-approximable relation} between distributive lattices: those are relations $ H \rightarrowtail C \times D$ with $C,D$ distributive lattices such that \begin{enumerate}
        \item for all $ c\leq c' \in C$ and $ d' \leq d$ in $D$, if one has $ H(y,x) $ then $H(x',y')$;
        \item for $c \in C$ and $(d_i)_{i \in I}$ a finite family in $D$, one has $ H(\bigvee_{i \in I} d_i, c) = \bigwedge_{i \in I} H(d_i, c) $ as 2-valued statement;
        \item for $ (c_i)_{i \in I}$ a finite family in $C$ and $d$ in $\D$, one has $ H(d, \bigwedge_{i \in I} c_i) = \bigwedge_{i \in I} H(d,c_i)$
        \item for $ (c_i)_{i \in I}$ a finite family in $C$ and $d \in D$, if one has $ H(d, \bigvee_{i \in I} c_i )$ then there is a finite decomposition $ d = \bigvee_{j \in J} d_j$ such that for each $j \in J$ there is $i \in I$ such that one has $ H(d_j,c_i)$.
    \end{enumerate}
    Then it is established at \cite{AbramskyJung}[theorem 7.2.26] that the category of distributive lattices and join-approximable relations is dually equivalent to the category of coherent spaces and \emph{all continuous maps} between them:
    \[ (\DLat_{\bigvee})^{\op} \simeq \Coh_{\Cont}  \]
    But each of the conditions in this definition corresponds obviously with one of the conditions in the definition of a continuous distributor of sites between coherent sites (where covers are finite jointly epimorphic families), and the duality result is analogous to the correspondence relating coherent distributors between pretopoi and arbitrary geometric morphisms between the associated coherent topoi.
\end{remark}

We end this subsection with an observation regarding geometric surjections. We defined cover-distributivity as a condition of joint preservation of cover along heteromorphisms; but there is a dual condition, with a reverse implication, which could be read as a joint cover reflection along heteromorphisms:

\begin{definition}
    A distributor $ H : \C \looparrowright \D$ between sites $ (\C,J), (\D,K)$ is said to be \emph{cover-testing} if for any sieve $S$ in $\C$ on an object $c$, if one has for any $d$ in $\D$ and any $ x \in H(d,c)$ that $ x^*H[S]$ is in $K(d)$, then $S$ is in $J(c)$. 
\end{definition}

Now recall that for a morphism of sites $ f : (\C,J) \rightarrow (\D,K)$ is \emph{cover-reflecting} if for any $S$ over $c$, if $f[S]$ is in $K(f(c))$, then $ S$ has to be in $J(c)$. In \cite{caramello2020denseness}[theorem 6.3 (i)] it is established that a morphism of sites $f$ induces a geometric surjection $\Sh(f)$ if and only if $ f$ is cover-reflecting. Similarly we have:

\begin{proposition}
    Let $H : (\C,J) \looparrowright (\D,K)$ be a distributor of sites; then $ \Sh(H) : \widehat{\D}_K \rightarrow \widehat{\C}_J$ is a geometric surjection if and only if $H$ is cover-testing. 
\end{proposition}

\begin{proof}
    Suppose that $ H$ is cover-testing; we show that $ \widehat{H}$ is cover-reflecting: let be $ S$ on $c$ generated by a covering family $(u_i : c_i \rightarrow c)_{i \in I}$; suppose that it is sent by $\widehat{H}$ to a jointly epimorphic family in $\widehat{\D}_K$ so that one has a local epimorphism $ \coprod_{i \in I} \widehat{H}(c_i) \twoheadrightarrow \widehat{H}(c) $ in $\widehat{\D}$. Hence for any $d$ in $\D$ and any $x \in H(d,c)$, there is a $K$-covering sieve $ (v_j : d_j \rightarrow d)_{j \in J}$ such that for any $ j \in J$, there is some $i \in I$ together with a factorization as below 
\[\begin{tikzcd}
	d & c \\
	{d_j} & {c_i}
	\arrow["x", squiggly, from=1-1, to=1-2]
	\arrow["{v_j}", from=2-1, to=1-1]
	\arrow["{\exists x_j}"', squiggly, from=2-1, to=2-2]
	\arrow["{u_i}"', from=2-2, to=1-2]
\end{tikzcd}\]
    but this latter condition says that the $K$-covering sieve generated from $ (v_j : d_j \rightarrow d)_{j \in J}$ is contained in $ x^*H[S]$ which is hence $K$-covering: hence for any $x \in H(d,c)$, $x^*H[S]$ is $K$-covering, hence $S$ is $J$-covering by the cover-testing property of $H$. Conversely, if $\Sh(H)$ is a surjection, then $\widehat{H}$ is cover-reflecting, and a reverse argument ensures that $H$ is cover-testing. 
\end{proof}

\subsection{Equipment-like properties of $\Site^{\mustrill}$}

Let us conclude this section with a few formal categorical observations. The bicategory $\Dist$ is the ur-example of the notion of \emph{pro-arrow equipment} relative to $\Cat$, as defined in \cite{Woodproarrow}. This question also relates to the Yoneda structure of $\Cat$ through the classifying property of the free cocompletion relative to distributors. As we are going to see, some of this structure is inherited by sites, although the dichotomy between morphisms and comorphisms breaks some other parts of this structure. 

\begin{division}
     It is well known that the 2-category $ \Dist$ is the \emph{Kleisli} 2-category of the relative lax-idempotent monad $ \widehat{(-)} : \Cat \rightarrow \CAT$. Although we are not going to enter in the detail of the technology underlying the notion of Kleisli 2-category of a 2-monad, let us examine how this statement refines for sites. While in absence of topology, the free cocompletion is only lax-idempotent and rises size issues, in the case of small generated sites, we actually can consider a true pseudomonad on the 2-category of sites together with morphisms of sites
    \[\begin{tikzcd}
	   {\Site^{\flat}} & {\Site^{\flat}}
	   \arrow["{\mathbb{T}}", from=1-1, to=1-2]
    \end{tikzcd}\]
    sending a site $ (\C,J)$ to the canonical site $ (\widehat{\C}_J, J_{can})$, which is again small generated. Hence, using that $J$-continuous functors $ \C \rightarrow \widehat{\D}_K$ are the same as morphisms of sites $ (\C,J) \rightarrow (\widehat{\D}_K, J_{can})$, we have
        \begin{align*}
        \Site^{\mustrill}_{\Cont}[(\C,J), (\D,K)] &\simeq \Cont_J[\C, \widehat{\D}_K] \\
        &\simeq \Site^{\flat}[(\C,J), (\widehat{\D}_K, J_{can})]
    \end{align*}
    which is exactly the property of the Kleisli 2-category of the pseudomonad $ \mathbb{T}$, for which we can state that
        \[ \Site^{\mustrill}_{\Cont} \simeq \mathbf{Kl}(\mathbb{T}) \]
        (Beware that we have to restrict to continuous distributors as sheaf topoi classify distributor of sites only up to sheafification.) This consideration hints at a possibility to examine analogs of the Yoneda structure of $\Cat$ in the context of sites, which will be the topic of a future work. 
\end{division}

\begin{division}[Equipments]
    The wide inclusion $ \Cat \rightarrow \Dist$ sending $ f : \C \rightarrow \D$ to the representable $ \D(1,f) : \C \looparrowright \D$ exhibits $ \Dist$ as an \emph{equipment}, in the sense of \cite{Woodproarrow}. Recall that for $ \mathcal{K}$ be a bicategory; then a pseudofunctor $ (-)_* : \mathcal{K} \rightarrow \mathcal{M}$ is said to \emph{equip $\mathcal{K}$ with proarrows} if \begin{enumerate}
        \item $(-)_*$ is a bijection on objects
        \item $ (-)_* $ is locally fully faithful
        \item for any arrow $f$ in $ \mathcal{K}$, $ f_*$ admits a right adjoint $ f^*$ in $\mathcal{M}$. 
    \end{enumerate}

    The main example is the inclusion $ \Cat \rightarrow \Dist$ sending $ f : \C \rightarrow \D$ to the representable distributor $ \D[1,f] : \C \looparrowright \D$, which we saw at \cref{representable distributors} to be left adjoint to $ \D[f,1] : \D \looparrowright \C$. Similarly, as established in \cite{woodtopos198471}, there is an equipment $ \Top \rightarrow \Top_\Lex^{\co} $ where $ \Top_\Lex$ is the bicategory of topoi with lex functors, sending $ f$ to its direct image part $f_* $: then $f^*$, as a left adjoint to $f_* $, is right adjoint in $\Top_\Lex^{\co}  $. 
\end{division}

\begin{division}
    Consisting of a bicategory of distributors, it is natural to expect from $ \Site^{\mustrill}$ to retain some of this structure relative to $ \Site^\flat$. Indeed the inclusion $ \Site^\flat \rightarrow \Site^{\mustrill}$ sending a morphism of sites $ f : (\C,J) \rightarrow (\D,K)$ to the representable $ \D[1,f] : (\C,J) \looparrowright (\D,K)$ is bijective on objects and locally fully faithful; however we see that the third axiom fails: for a morphism of site, there is no reason for the corepresentable $ \D[f,1] $ to be neither cover-flat nor cover-distributive, for which this right adjoint does not exist in $\Site^{\mustrill}$. On the other hand, we have a dual embedding $ (\Site^\sharp)^{\co} \rightarrow (\Site^{\mustrill})^{\op}$ sending a comorphism $f : (\D,K) \rightarrow (\C,J) $ to the corepresentable $ \D[f,1] : \C \looparrowright \D $, but again the candidate to be a right adjoint $ \D[1,f]$ fails to be a distributor of sites. However if $f$ happens to be both a morphism and a comorphism, then the adjunction $ \D[1,f] \simeq \D[f,1]$ exists in $ \Site^{\mustrill}$. However, for functors are seldom morphisms and comorphism of sites at once, we conjecture that a more relevant approach will be to relax the third axiom to study the statute of $ \Site^{\mustrill}$ to the pair $(\Site^\flat, \Site^\sharp) $.  
\end{division}   

Although the third axiom fails, some further requirements examined at \cite{Woodproarrow}[Axiom (C)] are still available. It is well known that distributors in $\Cat$ correspond with \emph{codiscrete cofibrations}, given by their gluing. Similarly, it appears that distributors of sites admit gluing and those gluings sit both in $\Site^\flat$ and $\Site^\sharp$.

\begin{division}[Gluing of a distributor]
    Let $H : \C \looparrowright \D$ be a distributor; then one can consider its \emph{gluing}, which is the category $ \Gl(H) $ having as objects the coproduct $ \Ob(\D) + \Ob(\C)$ and as morphism\begin{itemize}
        \item $\Gl(H)[(0,d),(0,d')] = \D[d,d']$
        \item $\Gl(H)[(1,c), (1,c')] = \C[c,c']$
        \item $\Gl(H)[d,c] = H(d,c)$ 
        \item $\Gl(H)[c,d] = \emptyset$
    \end{itemize}
    with composition provided by the functoriality of $H$ in the left and right variable, and identities provided by those of $\D$ and $\C$. Denote as $ \iota_{\C}$ and $\iota_{\D}$ the obvious inclusions. Those data are related in $ \Dist$ by a 2-cell
\[\begin{tikzcd}
	\C && \D \\
	& {\Gl(H)}
	\arrow["H", from=1-1, to=1-3]
	\arrow[""{name=0, anchor=center, inner sep=0}, "{\iota_\C}"', from=1-1, to=2-2]
	\arrow[""{name=1, anchor=center, inner sep=0}, "{\iota_\D}", from=1-3, to=2-2]
	\arrow["{\phi_H}"', between={0.2}{0.8}, Rightarrow, from=1, to=0]
\end{tikzcd}\]
    which is the transformation of distributors $ \Gl(H)[1, \iota_\D] \otimes H \Rightarrow \Gl(H)[1, \iota_\C]$ with component\begin{itemize}
        \item at $((0,d), c)$ given by the arrow 
    \[ \int^{d' \in \D} H(d',c) \times \Gl(H)[(0,d), (0,d')] \rightarrow \Gl(H)[(0,d), (1,c)] \]
    sending $[(x,v)]_\sim$ with $ x \in H(d',c)$ and $ v: d \rightarrow d'$ to the formal composite $ x(0,v)$ in $\Gl(H)$ 
        \item at $ ((1,c'),c)$ given by the initial map $ ! : \emptyset \rightarrow  \Gl[(1,c'), (1,c)]$ as each set of the form $ \Gl[(1,c'), (0,d')]$ is empty. 
    \end{itemize} 
    Then if $\C$ and $\D$ are endowed with topologies $J$ and $K$, one can equip $ \Gl(H)$ with a topology $J_H$ defined as follows:\begin{itemize}
        \item $J_H(0,d)$ is generated from sieves of the form $\iota_{\D}[R]$ in $K(d)$ together with those of the form $ x^*H[S]$ for $ S$ in $J(c)$ and $ x \in H(d,c)$: those sieves coincide with the pullback sieves $x^*\iota_{\C}[S]$ seein $x$ as an arrow $ x : (0,d) \rightarrow (1,c)$;
        \item $J_H(1,c) = J(c)$. 
    \end{itemize} 
    This topology makes both $ \iota_{\C}$ and $ \iota_{\D}$ trivially cover-preserving; moreover, $ \iota_{\C}$ is also cover-lifting, while $ \iota_\D$ is also trivially cover-flat relative to $J_H$; moreover, if $H$ is cover-distributing, then each $x^*H[S]$ is already in $K(d)$, for which in fact $ \iota_{\D}$ also becomes cover-lifting. If in addition $H$ is $K$-flat, then it is immediate that $ \iota_\C$ also becomes covering-flat for $ J_H$: this sums up as the following proposition:
\end{division}

\begin{proposition}\label{gluing}
    Let $ H : (\C,J) \looparrowright (\D,K)$ be a distributor of sites; then the inclusions $ \iota_{\C}$, $ \iota_{\D}$ are both morphisms and comorphisms of sites into $ (\Gl(H), J_H)$. Moreover, for a 2-cell as below 
\[\begin{tikzcd}
	\C && \D \\
	& \B
	\arrow["H", from=1-1, to=1-3]
	\arrow[""{name=0, anchor=center, inner sep=0}, "p"', from=1-1, to=2-2]
	\arrow[""{name=1, anchor=center, inner sep=0}, "q", from=1-3, to=2-2]
	\arrow["\phi"', between={0.2}{0.8}, Rightarrow, from=1, to=0]
\end{tikzcd}\]
    with $ (\B,L)$ a site and $p$, $ q$ morphisms of sites (resp. comorphisms of sites), then the unique functor $ \langle \phi \rangle : \Gl(H) \rightarrow \B$ such that $ \langle \phi \rangle * \phi_H = \phi$ is a morphism of sites (resp. a comorphism of sites). 
\end{proposition}

\begin{proof}
    From what precedes the gluing inclusions are both cover-preserving and lifting; they are hence comorphisms. Moreover, observe that $ \iota_\D$ is trivially cover-flat as all arrows left to $ \iota_{\D}$ come from $ \D$; on the other hand, in order to establish cover-flatness of $ \iota_\C$, it suffices to test relative to a cone of the form $(x_i : (0,d) \rightarrow (0,c_i)_{i \in I})$ for a finite diagram $(c_i)_{i \in I}$, but hence by cover-flatness of $H$ and \cref{cover-flatness for distributor synthetic form}, there is $R \in K(d)$ such that for any $v : d' \rightarrow d$ in $ R$ there exists a cone $ (u_{i} : c \rightarrow c_i)$ together with $ x' : d' \rightarrow c$ factorizing $ x$, but such local cone provides the desired data in $ \Gl(H)$. Hence both $\iota_\C, \iota_\D$ are simultaneously morphisms and comorphisms of sites.  

    Regarding the universal functor $ \langle \phi \rangle $, it sends a pair $ (0,d)$ to $ q(d)$, $ (1,c) $ to $ p(c)$ and acts accordingly on arrows on the left and right components of $ \Gl(H)$, while for a formal arrow $ x : (0,d) \rightarrow (1,c)$ we have
    \[ \langle \phi \rangle (x) = \phi_{(q(d),c)}([(x, \id_{q(d)})]_\sim)  \]
    where $\phi_{(q(d),c)}$ is the component of the transformation $ \phi : \B[1,q] \otimes H \Rightarrow \B[1,p]$. By definition, $ \langle \phi \rangle $ acts respectively on the two classes of objects $ (0,d)$ and $(1,c)$ as $q$ and $p$, and it is immediate that it inherits the property of being a morphism of sites (resp. a comorphism of sites) if both $p$ and $q$ share it. 
\end{proof} 

\begin{remark}
    This lemma says that the functors involved in the \emph{cotabulators} of a distributor that is a distributor of sites are both morphisms and comorphisms and are universal relative both to morphisms and to comorphisms.    
\end{remark}

\begin{remark}
    We can also consider the dual construction of the \emph{graph} $\Gr(H)$ of a distributor $H$, and in a way analogous to our observations regarding comma construction in the double category $ \Site^\natural$ of \cite{caramello2025morphisms}[subsection 2.4], the projections can be shown to inherit a part of the properties of $H$, but in this case the correspondence is less perfect than for the gluing construction. 
\end{remark}

\begin{division}
    Amongst further conditions for equipments, one also often includes the condition of local completeness and cocompleteness; this is for instance a property of $\Dist $ (see for instance \cite{benabou1973distributeurs}[proposition 2.3.4], as each homcategory $ \Dist[\C,\D]$ is both complete and cocomplete. However for distributors of sites, the situation is more delicate: the category $ \Site^{\mustrill}[(\C,J), (\D,K)]$ needs not be complete nor cocomplete in general ! However, one has the following result:
\end{division}

\begin{proposition}
    For any $(\C,J)$, $ (\D,K)$, $\Site^{\mustrill}_\Cont[(\C,J), (\D,K)] $ has filtered colimits. 
\end{proposition}

\begin{proof}
    From the equivalence $\Site^{\mustrill}_\Cont[(\C,J), (\D,K)] \simeq \Top[\widehat{\D}_K, \widehat{\C}_J]$ it is clear that $\Site^{\mustrill}_\Cont$ locally has filtered colimits since $\Top$ has. 
\end{proof}

%% file: Sections/Prequisite.tex
\subsection{Extensions and restrictions}


Here we recall some basics about extensions and nerves formulas which will be used recurrently in this work.

\begin{division}[Nerve and Yoneda extensions]\label{Nerve and Yoneda extensions}
Recall that any functor $ f: \mathcal{C} \rightarrow \mathcal{D}$ with $\mathcal{C}$ small and $\mathcal{D}$ locally small induces a \emph{nerve} functor 
\[\begin{tikzcd}
	{\mathcal{C}} & {\mathcal{D}} \\
	{\widehat{\mathcal{C}}}
	\arrow["f", from=1-1, to=1-2]
	\arrow["{\hirayo_\mathcal{C}}"', hook, from=1-1, to=2-1]
	\arrow[""{name=0, anchor=center, inner sep=0}, "{\mathcal{D}(f,1)}", from=1-2, to=2-1]
	\arrow["{\nu_f}", shorten >=2pt, Rightarrow, from=1-1, to=0]
\end{tikzcd}\]
sending $ d$ in $\mathcal{D}$ to the presheaf $ \mathcal{D}[f,d] : \mathcal{C}^{\op} \rightarrow \Set$, which comes equiped together with a 2-cell $ \nu_f$ whose component at $c$ is the transformation with component at $c'$ given as $f_{c',c} : \mathcal{C}[c',c] \rightarrow \mathcal{D}[f(c'), f(c)] $. On the other hand, if $ \mathcal{D}$ is cocomplete, this nerve functor admits a left adjoint given by the left Kan extension along the Yoneda embedding
\[\begin{tikzcd}
	{\mathcal{C}} & {\mathcal{D}} \\
	{\widehat{\mathcal{C}}}
	\arrow["f", from=1-1, to=1-2]
	\arrow["{\hirayo_\mathcal{C}}"', hook, from=1-1, to=2-1]
	\arrow[""{name=0, anchor=center, inner sep=0}, "{\lan_{\hirayo_\mathcal{C}}f}"', from=2-1, to=1-2]
	\arrow["\simeq"{description}, draw=none, from=1-1, to=0]
\end{tikzcd}\]
Then one can show that the nerve can also be exhibited as a left Kan extension through the relation 
\[ \mathcal{D}(f,1) = \lan_{f} \hirayo_{\mathcal{C}} \]    
\end{division}

\begin{division}[Extensions and restrictions along a functor]\label{Extensions and restrictions}
First recall that any functor $ f : \mathcal{C} \rightarrow \mathcal{D}$ induces a triple of adjoints 
\[\begin{tikzcd}
	{\widehat{\mathcal{C}}} && {\widehat{\mathcal{D}}}
	\arrow[""{name=0, anchor=center, inner sep=0}, "{\lext_f}", curve={height=-18pt}, from=1-1, to=1-3]
	\arrow[""{name=1, anchor=center, inner sep=0}, "{\rext_f}"', curve={height=18pt}, from=1-1, to=1-3]
	\arrow[""{name=2, anchor=center, inner sep=0}, "{\rest_f}"{description}, from=1-3, to=1-1]
	\arrow["\dashv"{anchor=center, rotate=-90}, draw=none, from=0, to=2]
	\arrow["\dashv"{anchor=center, rotate=-90}, draw=none, from=2, to=1]
\end{tikzcd}\]
where $ \lext_f $ (resp. $\rext_f$) sends a presheaf $ X : \mathcal{C}^{\op} \rightarrow \Set$ to its left (resp. right) Kan extension along $ f^{\op}$ as depicted below
\[\begin{tikzcd}
	{\mathcal{C}^{\op}} & \Set \\
	{\mathcal{D}^{\op}}
	\arrow["X", from=1-1, to=1-2]
	\arrow["{f^{\op}}"', from=1-1, to=2-1]
	\arrow[""{name=0, anchor=center, inner sep=0}, "{\lan_{f^{\op}}X}"', from=2-1, to=1-2]
	\arrow["{\zeta_{f^{\op}}}", shorten >=2pt, Rightarrow, from=1-1, to=0]
\end{tikzcd} \hskip1cm \begin{tikzcd}
	{\mathcal{C}^{\op}} & \Set \\
	{\mathcal{D}^{\op}}
	\arrow["X", from=1-1, to=1-2]
	\arrow["{f^{\op}}"', from=1-1, to=2-1]
	\arrow[""{name=0, anchor=center, inner sep=0}, "{\ran_{f^{\op}}X}"', from=2-1, to=1-2]
	\arrow["{\xi_{f^{\op}}}", shorten >=2pt, Rightarrow, from=0, to=1-1]
\end{tikzcd}\]
while the restriction functor $ \rest_f$ sends a presheaf $Y : \mathcal{D}^{\op} \rightarrow \Set$ to the precomposite
\[\begin{tikzcd}
	{\mathcal{C}^{\op}} \\
	{\mathcal{D}^{\op}} & \Set
	\arrow["{f^{\op}}"', from=1-1, to=2-1]
	\arrow["{\rest_f Y}", from=1-1, to=2-2]
	\arrow["Y"', from=2-1, to=2-2]
\end{tikzcd}\]
so that in particular for each $c$ in $\mathcal{C}$ one has 
\[ \rest_f Y(c) = Y(f(c)) \]

Both the extensions and restriction functors can also be constructed formally as follows: for $ f : \mathcal{C} \rightarrow \mathcal{D}$ one can construct two possible nerve functors, the left and right nerves, constructed respectively as the composite and the extension
\[\begin{tikzcd}
	{\mathcal{C}} & {\mathcal{D}} \\
	& {\widehat{\mathcal{D}}}
	\arrow["f", from=1-1, to=1-2]
	\arrow["{\mathcal{D}(1,f)}"', from=1-1, to=2-2]
	\arrow["{\hirayo_\mathcal{D}}", from=1-2, to=2-2]
\end{tikzcd} \hskip1cm
\begin{tikzcd}
	{\mathcal{C}} & {\mathcal{D}} \\
	{\widehat{\mathcal{C}}}
	\arrow["f", from=1-1, to=1-2]
	\arrow["{\hirayo_\mathcal{C}}"', hook, from=1-1, to=2-1]
	\arrow[""{name=0, anchor=center, inner sep=0}, "{\mathcal{D}(f,1)}", from=1-2, to=2-1]
	\arrow["{n_f}", shorten >=2pt, Rightarrow, from=1-1, to=0]
\end{tikzcd}\]
where $ \mathcal{D}(1,f)$ sends $c$ to the presheaf $ \mathcal{D}(1, f(c)) = \hirayo_{f(c)}$, while $\mathcal{D}(f,1)$ sends $ d$ in $\mathcal{D}$ to the presheaf $ \mathcal{D}(f, d) : \mathcal{C}^{\op} \rightarrow \Set$, which is actually the left Kan extension $ \mathcal{D}(f,1) = \lan_f \hirayo_{\mathcal{D}}$. But now one can compute either the left Kan extension $ \lan_{\hirayo_\mathcal{C}} D(1,f)$ on one side, while one can compute the left Kan extension of $\mathcal{D}(f,1)$ one the other side. 

\begin{proposition}\label{Formulas for restrictions and lextensions}
    For any functor $ f:  \mathcal{C} \rightarrow \mathcal{D}$ the nerve satisfies the following identities 
    \begin{align*}
    \rest_f &= \lan_{\hirayo_{\mathcal{D}}} \mathcal{D}(f,1) \\
     &=  \widehat{\mathcal{D}}(\hirayo_{\mathcal{D}}f, 1) \\
     &= \lan_{\mathcal{D}(1,f)} \hirayo_{\mathcal{C}}
    \end{align*}
while the left and right extensions can be computed as the extensions
    \[ \lext_f = \lan_{\hirayo_\mathcal{C}} \mathcal{D}(1,f) \hskip1cm \rext_f = \ran_{\hirayo_\mathcal{C}} \mathcal{D}(1,f) \]
 \end{proposition}

\begin{proof}
The first two expressions of the restriction functor are standards; for the third expression, observe that $ \mathcal{C}$ is cocomplete, so using that $ \hirayo_D f = \mathcal{D}(1,f)$ and applying the nerve expression of \cref{Nerve and Yoneda extensions} as an extension gives:
\[ \widehat{\mathcal{D}}(\mathcal{D}(1,f), 1) = \lan_{\mathcal{D}(1,f)} \hirayo_{\mathcal{C}} \]

For the extensions, observe that $\lan_{\mathcal{D}(1,f)} \hirayo_\mathcal{C}$ is right adjoint to the left extension $ \lan_{\hirayo_\mathcal{C}} \mathcal{D}(1,f)$; hence by uniqueness of the left adjoints, we have the expression of the left extension. 
\end{proof}

    
\end{division}

\begin{division}[Functoriality]

Extensions and restrictions are functorial relative to composition: 
\[ \lext_g\lext_f = \lext_{gf} \hskip1cm \rest_{gf} = \rest_f \rest_g \hskip1cm \rext_g\rext_f = \rext_{gf}\]

Suppose now one has a globular 2-cell
\[\begin{tikzcd}
	{\mathcal{C}} && {\mathcal{D}}
	\arrow[""{name=0, anchor=center, inner sep=0}, "f", curve={height=-12pt}, from=1-1, to=1-3]
	\arrow[""{name=1, anchor=center, inner sep=0}, "g"', curve={height=12pt}, from=1-1, to=1-3]
	\arrow["\phi", shorten <=3pt, shorten >=3pt, Rightarrow, from=0, to=1]
\end{tikzcd}\]
Then, the corresponding restriction functors are related as follows: $ \phi$ induces in a contravariant way a 2-cell between the nerves $ \mathcal{D}(\phi,1) : \mathcal{D}(g,1) \Rightarrow \mathcal{D}(f,1)$ which in turn induces a 2-cell $ \rest_g \Rightarrow \rest_f$ whose component at a presheaf $ Y : \mathcal{D}^{\op} \rightarrow \Set$ is the whiskering $ X* \phi^{\op} : Xg^{\op} \Rightarrow Xf^{\op}$. 

This 2-cell has a mate $ \lext_\phi : \lext_f \Rightarrow \lext_g$ (resp. $ \rext_\phi : \rext_f \Rightarrow \rext_g$), whose component at a presheaf $X$ is the universal 2-cell $ \lan_{f^{\op}} X \Rightarrow  \lan_{g^{\op}} X$ (resp. $ \ran_{f^{\op}} X \Rightarrow  \ran_{g^{\op}} X$) obtained from the universal property of the Kan extension from the composite 2-cell
\[\begin{tikzcd}[row sep=large]
	{\mathcal{C}^{\op}} && \Set \\
	{\mathcal{D}^{\op}}
	\arrow[""{name=0, anchor=center, inner sep=0}, "X", from=1-1, to=1-3]
	\arrow[""{name=1, anchor=center, inner sep=0}, "{f^{\op}}"', curve={height=12pt}, from=1-1, to=2-1]
	\arrow[""{name=2, anchor=center, inner sep=0}, "{g^{\op}}"{description}, curve={height=-12pt}, from=1-1, to=2-1]
	\arrow[""{name=3, anchor=center, inner sep=0}, "{\lan_{g^{\op}}}"', from=2-1, to=1-3]
	\arrow["{\zeta_g}", shorten <=2pt, shorten >=2pt, Rightarrow, from=0, to=3]
	\arrow["{\phi^{\op}}"', shorten <=5pt, shorten >=5pt, Rightarrow, from=2, to=1]
\end{tikzcd} \hskip1cm 
\begin{tikzcd}[row sep=large]
	{\mathcal{C}^{\op}} && \Set \\
	{\mathcal{D}^{\op}}
	\arrow[""{name=0, anchor=center, inner sep=0}, "X", from=1-1, to=1-3]
	\arrow[""{name=1, anchor=center, inner sep=0}, "{g^{\op}}"', curve={height=12pt}, from=1-1, to=2-1]
	\arrow[""{name=2, anchor=center, inner sep=0}, "{f^{\op}}"{description}, curve={height=-12pt}, from=1-1, to=2-1]
	\arrow[""{name=3, anchor=center, inner sep=0}, "{\ran_{f^{\op}}}"', from=2-1, to=1-3]
	\arrow["{\phi^{\op}}", shorten <=5pt, shorten >=5pt, Rightarrow, from=1, to=2]
	\arrow["{\xi_f}"', shorten <=2pt, shorten >=2pt, Rightarrow, from=3, to=0]
\end{tikzcd}\]
  
\end{division}